\documentclass[11pt]{amsart}
\usepackage{titletoc}
\usepackage[utf8]{inputenc}
\usepackage[T1]{fontenc}
\usepackage{indentfirst}
\usepackage[margin=1in]{geometry} 
\usepackage{lipsum}
\usepackage{amsmath,amscd,amsthm,amssymb,mathrsfs,stmaryrd,color,tikz-cd,graphicx,tikz,mathtools,subfigure}
\tikzcdset{arrow style=tikz,diagrams={>=stealth}}
\usepackage{enumerate}
\usepackage[bookmarks, bookmarksdepth=2, colorlinks=true, linkcolor=blue, citecolor=red, urlcolor=blue,pagebackref]{hyperref}
\usepackage{subfiles}
\renewcommand{\epsilon}{\varepsilon}

\newtheorem{theorem}{Theorem}[section]

\newtheorem{lemma}[theorem]{Lemma}
\newtheorem{corr}[theorem]{Corollary}

\newtheorem{proposition}[theorem]{Proposition}
\newtheorem{deff}[theorem]{Definition}

\newcommand{\bth}{\begin{theorem}}
\newcommand{\ble}{\begin{lemma}}
\newcommand{\bcor}{\begin{corr}}

\newcommand{\bdeff}{\begin{deff}}

\newcommand{\bprop}{\begin{proposition}}
\newcommand{\ele}{\end{lemma}}
\newcommand{\ecor}{\end{corr}}
\newcommand{\edeff}{\end{deff}}

\newcommand{\eprop}{\end{proposition}}

\newcommand{\la}{\lambda}

\newcommand{\e}{\varepsilon}

\newcommand{\supp}{\text{supp }}
\renewcommand{\Pi}{\varPi}
\renewcommand{\Re}{\rm{Re} \,}
\renewcommand{\Im}{\rm{Im} \,}
\renewcommand{\epsilon}{\varepsilon}
\newcommand{\sgn}{{\text {sgn}}}

\newcommand{\R}{{\mathbb R}}
\newcommand{\Z}{{\mathbb Z}}

\newcommand{\diag}{\Upsilon^{\text{diag}}}

\newcommand{\far}{\Upsilon^{\text{far}}}

\usetikzlibrary{matrix,arrows,decorations.pathmorphing,cd,topaths,calc,graphs}
\usepackage[all,cmtip]{xy}

\usepackage{relsize}

\renewcommand{\Re}{\operatorname{Re}}
\renewcommand{\Im}{\operatorname{Im}}

\makeatletter
\newcommand{\subalign}[1]{%
  \vcenter{%
    \Let@ \restore@math@cr \default@tag
    \baselineskip\fontdimen10 \scriptfont\tw@
    \advance\baselineskip\fontdimen12 \scriptfont\tw@
    \lineskip\thr@@\fontdimen8 \scriptfont\thr@@
    \lineskiplimit\lineskip
    \ialign{\hfil$\m@th\scriptstyle##$&$\m@th\scriptstyle{}##$\crcr
      #1\crcr
    }%
  }
}
\makeatother

\newcommand{\BC}{\mathbb{C}}

\newcommand{\BH}{\mathbb{H}}

\newcommand{\BQ}{\mathbb{Q}}
\newcommand{\BR}{\mathbb{R}}

\newcommand{\BZ}{\mathbb{Z}}

\newcommand{\bA}{\mathbf{A}}

\newcommand{\sA}{\mathscr{A}}
\newcommand{\bB}{\mathbf{B}}

\newcommand{\cD}{\mathcal{D}}

\newcommand{\sH}{\mathscr{H}}

\newcommand{\sM}{\mathscr{M}}

\newcommand{\sN}{\mathscr{N}}

\newcommand{\cO}{\mathcal{O}}

\newcommand{\cS}{\mathcal{S}}

\newcommand{\cT}{\mathcal{T}}

\newcommand{\fg}{\mathfrak{g}}

\newcommand{\vertiii}[1]{{\left\vert\kern-0.25ex\left\vert\kern-0.25ex\left\vert #1 
    \right\vert\kern-0.25ex\right\vert\kern-0.25ex\right\vert}}

\newcommand{\rom}[1]{\uppercase\expandafter{\romannumeral #1\relax}}

\renewcommand{\tilde}[1]{\widetilde{#1}}

\newcommand{\GL}{\operatorname{GL}}
\newcommand{\SL}{\operatorname{SL}}

\makeatletter
\newcommand{\colim@}[2]{%
  \vtop{\m@th\ialign{##\cr
    \hfil$#1\operator@font lim$\hfil\cr
    \noalign{\nointerlineskip\kern1.5\ex@}#2\cr
    \noalign{\nointerlineskip\kern-\ex@}\cr}}%
}
\newcommand{\Lim}{%
  \mathop{\mathpalette\varlim@{\leftarrowfill@\scriptscriptstyle}}\nmlimits@
}

\newcommand{\colim}{%
  \mathop{\mathpalette\varlim@{\rightarrowfill@\scriptscriptstyle}}\nmlimits@
}

\makeatother

\theoremstyle{definition}

\theoremstyle{remark}
\newtheorem*{rmk}{Remark}

\usepackage[new]{old-arrows}

\usepackage{amssymb}

\usepackage[utf8]{inputenc}

\title[$L^q$-norm bounds for arithmetic eigenfunctions]{$L^q$-norm bounds for arithmetic eigenfunctions via microlocal Kakeya-Nikodym estimate}
\author{Jiaqi Hou and Xiaoqi Huang}
\address{Department of Mathematics, Louisiana State University, Baton Rouge, LA 70803, USA}
\email{jhou7@lsu.edu}
\email{xhuang49@lsu.edu}
\date{}

\usepackage{hyperref}
\numberwithin{equation}{section}

\begin{document}

\begin{abstract}
    Let $X$ be a compact arithmetic congruence hyperbolic surface, and let $\psi$ be an $L^2$-normalized Hecke-Maass form on $X$ with sufficiently large spectral parameter $\lambda$. We give a new proof to obtain a power saving for the global $L^6$-norm $\|\psi\|_{L^6(X)}\lesssim_\varepsilon\lambda^{\frac{5}{36}+\varepsilon}$ over the local bound $\|\psi\|_{L^6(X)}\lesssim\lambda^{\frac{1}{6}}$ of Sogge. Our method uses a microlocal decomposition for $\psi$ and reduces the $L^6$-norm problem to microlocal Kakeya-Nikodym estimates for $\psi$, and we establish improved microlocal Kakeya-Nikodym estimates via arithmetic amplification developed by Iwaniec and Sarnak.
\end{abstract}

\maketitle

\section{Introduction}
Let $\Gamma$ be the group of the norm one units of a maximal order of an indefinite quaternion division algebra over $\BQ$ and let $\BH$ be the upper half plane.
The purpose of this paper is to use the amplification method to obtain improved $L^q$-norm estimates for an $L^2$-normalized Hecke-Maass form $\psi$, with spectral parameter $\la\gg1$, on the compact arithmetic congruence hyperbolic surface $X=\Gamma\backslash\BH$ for the critical exponent $q=6$, which in turn also yield improved estimates for all $q>2$ by interpolation.

\subsection{\texorpdfstring{$L^q$}{Lq}-norm bounds for Laplace eigenfunctions}
We first review the estimates of Laplace eigenfunctions on manifolds.
Let $(M,g)$ be a  compact $n$-dimensional Riemannian manifold without boundary, 
$\Delta_g$ be the Laplace-Beltrami operator associated with the metric $g$ on $M$. We let $0=\la_0^2<\la^2_1\le \la^2_2\le \cdots$ denote the eigenvalues labeled with respect to the multiplicity of ${-\Delta_g}$ and $e_{\la_j}$ the associated $L^2$-normalized eigenfunctions, that is,
\begin{equation}\label{1.2}
(\Delta_g+\la^2_j )e_{\la_j}=0, \quad
\text{and } \quad \int_M |e_{\la_j}(x)|^2 \, dx=1.
\end{equation}

In Sogge~\cite{sogge881} (see also Avakumovi\'c \cite{Avakumovic} and Levitan \cite{Levitan} when $q = \infty$), the universal bound for $q\ge2$
\begin{equation}\label{r1}
\|e_\la\|_{L^q(M)}\lesssim \la^{\delta(q,n)}
\end{equation}
was obtained. Here $\delta(q,n)$ is given by
\begin{equation*}
            \delta(q,n) = \begin{dcases}
                \tfrac{n-1}{4} - \tfrac{n-1}{2q} \;&\text{ for }\;2\leq q\leq q_c,\\
                \tfrac{n-1}{2} - \tfrac{n}{q} \;\;\;\;&\text{ for }\;q_c\leq q\leq\infty,
            \end{dcases}
            \quad\text{ and }\quad q_c=\tfrac{2(n+1)}{n-1}.
\end{equation*}
The result also applies to quasimodes that are linear combinations of eigenfunctions whose eigenvalues lie in the unit interval $[\la,\la+1]$. 
When $M$ is a compact locally symmetric space of higher rank, besides the Laplace-Beltrami operator, one can consider eigenfunctions of the full ring of invariant differential operators on $M$. The general local bound was obtained by Sarnak \cite{sarnak2004letter} (for $q=\infty$) and Marshall \cite{Mar16HigerRank} (for any $q\ge2$ with some log losses at the corresponding critical exponents), which extends and improves the bound \eqref{r1}. 

The estimate \eqref{r1} is sharp on the standard round sphere $\mathbb{S}^n$ by taking $e_\la$ to be the zonal functions or the highest weight spherical harmonics, which are point-focusing and geodesic-focusing eigenfunctions, respectively. Under additional geometric assumptions--such as nonpositive sectional curvature or on flat tori, one can improve \eqref{r1}. See e.g., the work of Bourgain-Demeter \cite[Section 2.3]{BoDe}, as well as  Germain-Myerson \cite{germain2022bounds}, and the reference therein for related results on the torus. On compact manifolds with non-positive sectional curvature,
$(\log\la)^{-\frac12}$ improvements to \eqref{r1} was obtained in B\'erard \cite{Berard} for $q=\infty$, and by Hassell and Tacy \cite{HassellTacy} for $q>q_c=\tfrac{2(n+1)}{n-1}$. The same improvement also holds, more generally, for spectral projection operators onto the shorter interval $[\la,\la+(\log\la)^{-1}]$. These results were further generalized under weaker dynamical assumptions in the recent work of Canzani-Galkowski \cite{CGGrowth, canzani2023weyl}. 

In the other direction, for $q\le q_c$ similar improvements were obtained in Huang--Sogge \cite{huang2024curvature}. More generally, they proved sharp quasimode estimates for logarithmic quasimodes, which were also used to characterize compact manifolds with constant curvature. See also Blair and Sogge \cite{blair2018concerning,SBLog} and Blair, Huang, and Sogge \cite{blair2024improved} for related earlier work in this direction.

We also note that when $q<q_c$, for a general compact manifold without boundary, a stronger version of \eqref{r1} for $L^2$-normalized eigenfunctions was obtained in Blair and Sogge \cite{blair2015refined} 
\begin{equation}\label{r1a}
\|e_\la\|_{L^q(M)}\lesssim \la^{\delta(q,n)}\vertiii{e_\la}^{\frac{q_c}{q}-1}_{KN},\quad\text{ where }\; \tfrac{2(n+2)}{n}<q<q_c.
\end{equation}
Here the Kakeya–Nikodym norm is defined by 
\[
\vertiii{e_\la}^{2}_{KN}
= \sup_{\gamma \in \Pi}
\int_{\mathcal{T}_{\la^{-1/2}}(\gamma)}
|e_\la|^2 \, dV_g,
\]
where $\Pi$ denotes the set of unit-length geodesic segments in $M$ and ${\mathcal{T}_{\la^{-1/2}}(\gamma)}$ denote the $\la^{-1/2}$ neighborhood of a given geodesic $\gamma$ segment.  The Kakeya–Nikodym norm is closely related to the Nikodym maximal operator on manifolds (as well as its Euclidean counterpart, the Kakeya maximal operator), both of which have been extensively studied due to their central role in harmonic analysis. See, for example, \cite{sogge2011kakeya} and the references therein for further background.

The estimate \eqref{r1a} generalizes earlier related work of Bourgain \cite{bourgain2009geodesic}, Sogge \cite{sogge2011kakeya}, and Blair and Sogge \cite{BlairSoggeKaknik,BlairSoggeRefined}. The proofs rely fundamentally on bilinear oscillatory integral estimates in harmonic analysis.
More recently, \eqref{r1a} was extended to the larger range $\tfrac{2(3n+1)}{3n-3}\le q<q_c$ in dimensions $n\ge 3$ by Gao, Wu and Xi \cite{gao2025sharp}, up to a loss of $\la^\e$. Their argument builds on recent advances in Fourier restriction theory, including refined decoupling estimates.

\subsection{\texorpdfstring{$L^q$}{Lq}-norm bounds for Hecke-Maass forms}
The paper concerns improvements of \eqref{r1} for arithmetic eigenfunctions on hyperbolic surfaces. Recall that $X=\Gamma\backslash\BH$ is an arithmetic congruence hyperbolic surface, where $\Gamma\subset \SL(2,\BR)$ is taken to be the norm one units of a maximal order of an indefinite quaternion division algebra over $\BQ$ (see Section \ref{sec: arith hyper sur} for the precise definitions) so that $X$ is compact. Let $\psi$ be an $L^2$-normalized Hecke–Maass form on $X$. Namely, $\int_X|\psi(x)|^2dx=1$ and $\psi$ is a joint eigenfunction of the Laplace-Beltrami operator $\Delta_g$ and Hecke operators on $X$. Let $\la$ be the spectral parameter of $\psi$, so that $\Delta \psi+(\frac14+\lambda^2)\psi=0$. As we are considering large-eigenvalue asymptotics, we will also assume that $\la$ is real and $\lambda\gg1$. Iwaniec and Sarnak \cite{IS95} made a Lindel\"of hypothesis type conjecture that $\|\psi\|_{L^\infty(X)}\lesssim_\epsilon\la^{\epsilon}$, which is strong and widely open. The conjecture implies that $\|\psi\|_{L^q(X)}\lesssim_\epsilon\la^{\epsilon}$ for any $q>2$.

For these $X$ and $\psi$, the first power-saving result was obtained by Iwaniec and Sarnak \cite{IS95}. They showed that 
\begin{align}\label{eq: IS bound}
    \|\psi\|_{L^\infty(X)} \lesssim_\epsilon\lambda^{\frac{5}{12}+\epsilon},
\end{align}
which is a $\la^{\frac1{12}}$ improvement over the uniform bound \eqref{r1} for $n=2$ and $q=\infty$. In two-dimensional cases, the critical exponent is $q_c=6$, so, by interpolation, \eqref{eq: IS bound} implies an improvement for  $ \| \psi \|_{L^q(X)}$ for all $q>6$. Their proof uses the technique known as arithmetic amplification. This approach has since been used by many authors to bound sup-norms of Hecke-Maass forms on other groups. For instance, see \cite{Van97,BM13,Mar14,BHM16,BM16,BP16,BHM20} for results in the spectral aspect.

For $q<6$, Marshall \cite{Mar16} studied the related $L^2$ geodesic restriction problem for the Hecke-Maass form $\psi$ and obtained a similar $\la$-power improvement over the uniform result of Burq,  G{\'e}rard and Tzvetkov \cite{BGTrestr}. We let $\Pi$ denote the set of all unit-length geodesic segments on $X$. Marshall obtained the uniform bound
\begin{align}\label{eq: Marshall geodesic bound}
    \sup_{\gamma\in\Pi} \| \psi|_\gamma \|_{L^2(\gamma)}\lesssim_\epsilon \lambda^{\frac{3}{14}+\epsilon}.
\end{align}
He reduced $L^2$-norm bounds for $\psi$ along $\gamma\in\Pi$ to bounds for various
Fourier coefficients along $\gamma$.
His improvements over the local bounds extend the technique of arithmetic amplification developed by Iwaniec-Sarnak.
Using the result of Blair and Sogge in  \eqref{r1a}, the geodesic restriction bound \eqref{eq: Marshall geodesic bound} implies that
\begin{align}\label{eq: Marshall L4}
    \| \psi \|_{L^4(X)}\lesssim_\epsilon \la^{\frac{1}{8}-\frac{1}{56}+\epsilon}.
\end{align}
Moreover, using \eqref{eq: Marshall geodesic bound} and \eqref{r1a} also yields an improvement over \eqref{r1} for all $2<q<6$ by interpolation. See \cite{hou2024restrictions,hou2025kakeya} for the extension of this method to improve bounds for Kakeya-Nikodym norms on other groups of rank one. For the higher rank group $\SL(3,\BR)$, \eqref{eq: Marshall geodesic bound} was extended by Marshall \cite{Mar15} to $L^2$ maximal flat restrictions of Hecke-Maass forms. 

It would be interesting to extend the above results and obtain a power saving over \eqref{r1} for the critical exponent $q_c=6$. A much stronger $L^4$-norm bound than \eqref{eq: Marshall L4} was obtained by Humphries and Khan \cite[Theorem 1.4]{HK25}:
\begin{align}\label{eq: HK L4}
    \| \psi \|_{L^4(X)}\lesssim_\epsilon \la^{\frac{3}{152}+\epsilon}.
\end{align}
They applied Parseval’s identity to $|\psi|^2$ and used the Watson–Ichino triple product formula to reduce bounds for $\| \psi \|_{L^4(X)}$ to bounds for certain moments of $L$-functions. By interpolating with the sup-norm bound \eqref{eq: IS bound} of Iwaniec-Sarnak, Humphries and Khan got
\begin{align}\label{eq: HK L6}
    \| \psi \|_{L^6(X)}\lesssim_\epsilon \la^{\frac{26}{171}+\epsilon}=\la^{\frac{1}{6}-\frac{5}{342}+\epsilon}.
\end{align}

On the noncompact quotient $\SL(2,\BZ)\backslash\BH$, for a Hecke-Maass cusp form $\psi$ whose spectral parameter is $\la$, both the sup-norm bound \eqref{eq: IS bound} and the $L^4$-norm bound \eqref{eq: HK L4} hold as well, so one still has the same $L^6$-norm bound \eqref{eq: HK L6}. Ki \cite{ki20234} showed an optimal bound $\|\psi\|_{L^4(\SL(2,\BZ)\backslash\BH)}\lesssim_\epsilon\la^{\epsilon}$, by using the
Fourier-Whittaker expansions of Hecke-Maass cusp forms. This result justifies Iwaniec-Sarnak's conjecture for $2<q\leq4$, and therefore implies a better $L^6$-norm bound
\begin{align}\label{eq: Ki L6}
    \| \psi \|_{L^6(\SL(2,\BZ)\backslash\BH)}\lesssim_\epsilon \la^{\frac{1}{6}-\frac{1}{36}+\epsilon}.
\end{align}

We prove the following improved $L^6$-norm bound over \eqref{eq: HK L6} in the \textit{compact} setting.
\begin{theorem}\label{thma}
    Let $\psi$ be an $L^2$-normalized Hecke–Maass form on $X$ with spectral parameter $\lambda\gg1$. We have
    \begin{equation*}
        \| \psi \|_{L^6(X)} \lesssim_\epsilon
        \lambda^{\frac{1}{6}-\frac{1}{36}+\epsilon}.
    \end{equation*}
\end{theorem}

\subsection{Proof strategy}

A main step in the proof of Theorem~\ref{thma} is the following.
\begin{theorem}\label{thma1}
    Let $(M,g)$ be a compact $2$-dimensional Riemannian manifold without boundary. Let $e_\la$   be an $L^2$-normalized eigenfunction of  $-\Delta_g$ with eigenvalue $\la^2$, satisfying \eqref{1.2} for some $\la\gg1$. Set $\theta_0=\la^{-\delta_0}$ and let $A_\nu^{\theta_0}$ be the pseudo-differential operator defined in \eqref{qnusymbol}. Suppose that there exists $0<\delta_\infty<\eta<\frac{1}{2}$ so that, $\|e_\la\|_{L^\infty(M)}\lesssim_\e \la^{\frac{1}{2}-\delta_\infty+\e}$, and, for all $0<\delta_0<\frac12$,
    \begin{equation}\label{n3'}
        \sup_\nu\|A_\nu^{\theta_0} e_\la\|_{L^\infty(M)}\lesssim_\e \la^{(\frac{1}{2}-\delta_\infty)(1-\delta_0)+\e},
    \end{equation}
    and 
    \begin{equation}\label{n30}
        \sup_\nu\|A_\nu^{\theta_0} e_\la\|_{L^2(M)}\lesssim_\e \la^{-
        \eta\delta_0+\e},
    \end{equation}
 Then for any $\theta_0=\la^{-\delta_0}$ with $0<\delta_0\leq\frac{3}{8}(1-{\delta_\infty})$, we have 
 \begin{equation}\label{j5}
       \|e_\la\|_{L^6(M)}\lesssim_\e\la^{\frac{1}{6}-\frac{1}{3}\delta_\infty+\e}+  \la^{\frac{1}{6}-\frac{2}{3}\eta\delta_0+\e} .
    \end{equation}
\end{theorem}

For each fixed $\nu$, the operator $A_\nu^{\theta_0}$ is microlocally supported in a $\theta_0$-neighborhood of a geodesic segment and is frequency-localized to directions near that geodesic; the precise definition will be given in Section~\ref{sec: microlocal decomposition}. We shall refer to estimates for $\sup_\nu \| A_\nu^{\theta_0}e_\la\|_{L^q(M)}$ ($2\le q\le\infty$) as \textit{microlocal Kakeya--Nikodym estimates} for $e_\lambda$, as it incorporates additional frequency localization compared to the Kakeya--Nikodym norm appearing in \eqref{r1a}. Moreover, the width of the associated geodesic neighborhood is $\lambda^{-\delta_0}$, which may be larger than the $\lambda^{-1/2}$-tubes appearing in \eqref{r1a}. 

The proof of Theorem~\ref{thma1} occupies Section \ref{sec: 2}. It starts with a strategy similar to that in \cite{huang2024curvature}, combined with a multi-scale iteration argument that exploits both \eqref{n3'} and \eqref{n30}. This iteration scheme is similar in spirit to the induction-on-scales approach in harmonic analysis.
The main idea is to decompose $e_\lambda$ in phase space as a sum of microlocally localized pieces $A_\nu^{\theta_0} e_\lambda$, and to expand $e_\lambda^2$ as a bilinear sum involving products of the form $A_\nu^{\theta_0} e_\lambda \, A_{\nu'}^{\theta_0} e_\lambda$. The \emph{broad} term—namely, pairs that are well separated in phase space—is handled using bilinear oscillatory integral estimates together with sup-norm bounds; these account for the first term on the right-hand side of \eqref{j5}.
The \emph{narrow} term, corresponding to nearby pairs $A_\nu^{\theta_0}$ and $A_{\nu'}^{\theta_0}$, satisfies an almost orthogonality property in $\nu$. This is treated via an iteration argument. At each step of the iteration, the broad term is again estimated using bilinear bounds together with the refined sup-norm estimate \eqref{n3'} and the refined $L^2$ bound \eqref{n30}, yielding improvements over the first term in \eqref{j5}.
After finitely many iterations, we terminate the procedure and apply a different argument to handle the remaining narrow contributions, which give rise to the second term on the right-hand side of \eqref{j5}.

The assumption on $\delta_0$, which depends on $\delta_\infty$, is chosen for convenience to neglect certain error terms arising from commutator estimates between the operator \(A_\nu^{\theta_0}\) and the spectral projection operators used in the analysis. Although this assumption could likely be improved with a more refined argument, the present range is sufficient for our applications.

The $L^6$-norm does not play an essential role in the proof of Theorem~\ref{thma1}. In fact, using the same ideas as above along with an iteration argument, one can show that
\begin{equation*}
    \|e_\la\|_{L^q(M)} \lesssim_\e \la^{(1-\frac{4}{q})(\frac{1}{2}-\delta_\infty)+\e} + \la^{\frac{1}{4}-\frac{1}{2q}-\frac{q-2}{q}\eta\delta_0+\e},
\end{equation*}
for all $q \in [q_0,6]$, where $q_0 \approx 5.8$ is a fixed exponent. This would yield a corresponding extension of Theorem~\ref{thma}, namely that
\[
\| \psi \|_{L^q(X)} \lesssim_\epsilon \la^{(1-\frac{4}{q})\frac5{12}+\e}, \quad\forall q_0\le q\leq\infty.
\]
For simplicity, we do not pursue this extension here.

To apply Theorem \ref{thma1}, we need the following improved microlocalized Kakeya-Nikodym estimates for $\psi$, which rely on the arithmetic assumptions on $X$ and $\psi$.

\begin{theorem}\label{thm: improved KN norm}
    Let $\psi$ be an $L^2$-normalized Hecke–Maass form on $X$ with spectral parameter $\lambda\gg1$. Set $\theta_0=\la^{-\delta_0}$ with $0<\delta_0<\frac12$, and let $A_\nu^{\theta_0}$ be the pseudo-differential operator defined in  \eqref{qnusymbol}. We have
    \begin{equation}\label{i3}
        \sup_\nu\|A_\nu^{\theta_0} \psi\|_{L^2(X)}\lesssim_{\delta_0,\e }\la^{-\frac{\delta_0}{8}+\e},
    \end{equation}
    and
    \begin{equation}\label{eq: improved KN sup}
        \sup_\nu\|A_\nu^{\theta_0} \psi\|_{L^\infty(X)}\lesssim_{\delta_0,\e} \la^{\frac{5}{12}(1-\delta_0)+\e}. 
    \end{equation}
\end{theorem}

We will prove Theorem \ref{thm: improved KN norm} in Section \ref{sec 3} by the method of arithmetic amplification from \cite{IS95}. 
Similar to usual amplification arguments, we have counting problems and analytic problems. The counting problem (see Lemma \ref{lem: Hecke return new}) for \eqref{i3} is to estimate the number of times a Hecke operator maps $\gamma$ back close to itself. Using the microlocal support of $A_\nu^{\theta_0}$, we can reduce the analytic problem to an oscillatory integral estimate appearing in \cite{Mar16}, which will be proved in Proposition \ref{prop: bound for I_A}.

Using the operator bound $\sup_\nu\| A_\nu^{\theta_0}\|_{L^\infty\to L^\infty}\lesssim1$, it follows from Iwaniec-Sarnak's bound \eqref{eq: IS bound} that $\sup_\nu\|A_\nu^{\theta_0} \psi\|_{L^\infty(X)}\lesssim_\e \la^{\frac{5}{12}+\e}$.
We expect to control the sup-norm of $A_\nu^{\theta_0} \psi$ better because it behaves like a Gaussian beam along a geodesic, which usually has a smaller sup-norm. We believe the local sup-norm bound for $A_\nu^{\theta_0} \psi$ is $\la^{\frac{1-\delta_0}{2}}$ by Proposition \ref{p}. Because of the microlocalization, the kernel function in the amplified pretrace formula has an extra rapid decay feature, compared to the standard kernel estimate in the sup-norm problem, when the group translation separates the central geodesic in the support of $A_\nu^{\theta_0}$. We establish this kernel estimate in Proposition \ref{p}. Hence we have a new counting problem by studying the simultaneous point and geodesic returns. This extra constraint allows us to prove a power improvement over Iwaniec-Sarnak's counting bound. Our new counting bound is Proposition \ref{prop: simultaneous return} and will be proved in Section \ref{sec: counting}.

\begin{proof}[Proof of Theorem \ref{thma}]
    We prove this by assuming Theorems \ref{thma1} and \ref{thm: improved KN norm}. In particular, by combining \eqref{j5} with \eqref{i3} and \eqref{eq: improved KN sup} and choosing $\delta_\infty=\frac{1}{12}$ and $\eta=\frac{1}{8}$, we get
    \[
    \|\psi\|_{L^6(X)}\lesssim_\e\la^{\frac{1}{6}-\frac{1}{36}+\e}+  \la^{\frac{1}{6}-\frac{\delta_0}{12}+\e}.
    \]
    We complete the proof by choosing $\delta_0=\frac{1}{3}$.
    Note that the requirement $0<\delta_0<\frac{3}{8}(1-\delta_\infty)=\frac{11}{32}$ is satisfied.
\end{proof}

    In Theorem \ref{thma}, the arithmetic assumption that $\psi$ is a Hecke eigenfunction is only used to invoke the bound Iwaniec-Sarnak \cite{IS95} and to prove Theorem \ref{thm: improved KN norm}, both of which only rely on arithmetic amplification. Hence, the maximal order in the theorem can be replaced with an Eichler order as in \cite{IS95}. Moreover, we use only unramified Hecke operators, so we can assume $\psi$ to be an eigenfunction under even fewer Hecke operators as long as the amplification argument can be adapted.
    
    The strategy developed here can also be generalized to study the $L^{q_c}$-norm problems for Hecke-Maass forms on compact locally symmetric spaces of rank one, as long as the corresponding sup-norm problem and microlocal Kakeya-Nikodym problem can be solved by the arithmetic amplification technique; we will explore this direction in future work. Theorem \ref{thma1} is expected to be generalized to Laplace-Beltrami eigenfunctions on compact Riemannian manifolds of any dimension. For example, when the manifold is taken to be a compact arithmetic hyperbolic 3-manifold (the critical exponent is $q_c=4$), the corresponding sup-norm problem for Hecke-Maass forms is expected to be solved (see e.g. \cite{BHM16} for noncompact cases). The corresponding microlocal Kakeya-Nikodym estimates should be established in a similar way as Theorem \ref{thm: improved KN norm}.

\subsection{Notation}
Throughout the paper, the notation $A\lesssim B$ means that there is a positive constant $C$ such that $|A| \leq C B$, and $A\sim B$ means that $A\lesssim B\lesssim A$. We also use $A =O(B)$ to mean $A\lesssim B$. The notation  $A\gg B$ means there is a sufficiently large positive constant $C$ such that $A\ge CB$, and similarly for $A\ll B$.

If $f\in L^1(\BR^n)$, the Fourier transform $\hat{f}$ of $f$ in this paper is defined as
\begin{align*}
    \hat{f}(\xi)=\int_{\BR^n} f(x)e^{-i\langle x,\xi\rangle} dx. 
\end{align*}
Here $x=(x_i),\xi=(\xi_i)\in\BR^n$ and $\langle x,\xi\rangle=\sum_i x_i\xi_i$.
If $f\in\cS(\BR^n)$ is a Schwartz function, the Fourier inversion formula reads
\[
    f(x) = (2\pi)^{-n}\int_{\BR^n}\hat{f}(\xi)e^{i\langle x,\xi\rangle} d\xi.
\]
We say a symbol $p(x,\xi)$ is in the class $S^m_{\rho,\delta}$ if\[
|\partial_x^\beta\partial^\alpha_\xi p(x,\xi)|\lesssim_{\alpha,\beta} \langle\xi\rangle^{m-\rho|\alpha|+\delta|\beta|},
\]
where $\langle\xi\rangle=(1+|\xi|^2)^{1/2}$ is the Japanese bracket.
The pseudo-differential operator $p(x,D)$ is defined by the integral
\[
p(x,D)f(x)=(2\pi)^{-n}\int p(x,\xi)\hat{f}(\xi) e^{i\langle x,\xi\rangle}d\xi=(2\pi)^{-n}\iint p(x,\xi)f(y) e^{i\langle x-y,\xi\rangle}d\xi dy.
\]
We say $P$ is the pseudo-differential operator with the compound symbol $p(x,y,\xi)$ if
\[Pf(x)=(2\pi)^{-n}\iint p(x,y,\xi)f(y) e^{i\langle x-y,\xi\rangle}d\xi dy.\]

\subsection{Acknowledgements}
The authors would like to thank Farrell Brumley, Simon Marshall, Lior Silberman, Christopher D. Sogge, and Radu Toma for helpful conversations or comments.  The second author was supported in part by the Simons Foundation and NSF (DMS-2452860). 
\bigskip

\section{Reduction to microlocal Kakeya-Nikodym estimates}\label{sec: 2}
In this section, we prove Theorem \ref{thma1}. The constructions used here are similar to those in Section 2.2 of \cite{huang2024curvature}. Let $(M,g)$ and $e_\la$ be as in the theorem.
In order to use the local harmonic analysis tools related to oscillatory integrals, it is convenient to use the 
smooth spectral projection operators
of the form
\begin{equation}\label{2.1}
\sigma_\la =\rho(\la-P),  \quad P=\sqrt{-\Delta_g},
\end{equation}
where 
\begin{equation}\label{2.2}
\rho\in {\mathcal S}(\R), \, \, \rho(0)=1 \, \, \,
\text{and } \, \, \text{supp }\Hat \rho\subset \delta\cdot [1-\tilde \delta,1+\tilde\delta]
=[\delta-\tilde\delta\delta, \delta+\tilde \delta\delta].
\end{equation}
Here $\tilde \delta,\delta$ are some small positive constants. Note that the condition $\rho(0)=1$ implies $\sigma_\la e_\la=\rho(0)e_\la=e_\la$. 
Moreover, by \cite[Lemma 5.1.3]{SFIOII},
\begin{equation}\label{2.2a}
\begin{aligned}
        \sigma_\la f(x)&=(2\pi)^{-1}\int^{2\delta}_{-2\delta} \Hat \rho(t) e^{it\la} e^{-itP}f\, dt
     \\ &=
 \lambda^{1/2}\int_Me^{i\lambda d_g(x,y)}a(x,y,\lambda) f(y)dy +R_\lambda f (x)
\end{aligned}
 \end{equation}
 where 
 $$ \|R_\lambda f\|_{L^\infty(M)}\le C_N\lambda^{-N} \|f\|_{L^1(M)}\quad\forall N=1,2,3..
 $$
and $a\in C^\infty$ has the property that 
\begin{align}\label{eq: derivative bound for a}
\begin{split}
    &a(x,y,\la)=0 \qquad\text{ if }\,\, d_g(x,y)\notin (\delta/2, 2\delta),\\
    \text{ and }\;&\partial_{x,y}^\gamma a(x,y,\la)\lesssim_\gamma 1\qquad\text{for any multi-index } \gamma.
\end{split}
\end{align}
Here we shall need that $\delta$ is smaller than the injectivity radius of $M$ and it also must be
chosen small enough so that the phase function $d_g(x,y)$ satisfies the Carleson-Sj\"olin condition when $d_g(x,y)\in (\delta/2, 2\delta)$,
with $d_g(\, \cdot \, ,\, \cdot \, )$ denoting the Riemannian distance function. The additional constant $\tilde\delta$ is required to be sufficiently small to guarantee that the bilinear oscillatory integral estimates used in \cite{huang2024curvature} are valid.

Let us further localize using microlocal cutoffs. 
Let $\{\phi_i(x,\xi)\}_{j=1}^{N_0}$ be a finite partition of unity of $S^*M$, where each $\phi_i$ is supported in a sufficiently small neighborhood of some $(x_i,\xi_i)\in S^*M$. Here $N_0$ is a large fixed constant independent of $\la$. Fix a function
$$\beta\in C^\infty_0((1/2,2)),   \, \, \, \text{and}\,\,\,
\beta(\tau)=1 \, \, \, \text{for } \, \tau \, \, \text{near } \, \, 1, \, \,$$
Define \begin{equation}\label{2.8}
  B_{i,\la}(x,\xi)=  \phi_i(x,\xi/p(x,\xi))\beta(p(x,\xi)/\la)
\end{equation}
where $p(x,\xi)=|\xi|_g=\sqrt{g^{ij}(x)\xi_i\xi_j}$ denotes the principal symbol of $P=\sqrt{-\Delta_g}$. We may assume the $x$-support of each symbol $\phi_i(x,\xi)$ is contained in a local coordinate chart $\Omega_i$, and let 
  $B_{i,\la}(x,D)$ be the pseudo-differential operator defined in this chart with symbol $B_{i,\la}(x,\xi)$ (see e.g., Chapter 3.3 in \cite{SFIOII} for the definition of pseudo-differential operators on compact manifolds). Let ${B}_{i,\la}(x,y)$ denote the Schwartz kernel of  $B_{i,\la}(x,D)$. We may further assume (after inserting a cutoff in the $y$ variable) that
\[
B_{i,\lambda}(x,y)=B_{i,\lambda}(x,y)\,\tilde\phi_i(y),
\]
where $\tilde\phi_i\in C_0^\infty(\Omega_i)$ satisfies $\tilde\phi_i\equiv1$
in a neighborhood of the $x$–support of $\phi_i(x,\xi)$.
Then, 
by repeated integration by parts, we have
$$    {B}_{i,\la}(x,y)=
O_N\bigl(\la^{2} (1+\la|x-y|)^{-N}\bigr) \qquad\text{ for any }N>0.
$$
Consequently, by Young’s inequality,
 $B_{i,\la}(x,D)$ 
are uniformly bounded on $L^p(M)$, i.e.,
\begin{equation}\label{2.9}
\|B_{i,\la}(x,D)\|_{L^p\to L^p}=O(1) \quad \text{for } \, \, 1\le p\le \infty.
\end{equation}
Since $p(x,\nabla_xd_g(x,y))=1$ (i.e. the eikonal equation for the distance function), by \eqref{2.2a} and an integration by parts argument, we have
\begin{equation}\label{2.12}
\, \, \bigl\|\sigma_\la -\sum_{i=1}^{N_0}B_{i,\la}(x,D)\circ \sigma_\la\bigr\|_{L^2\to L^q}=O_N(\la^{-N}), \qquad \text{ for any }N>0 \,\,\,\text{and}\,\,\,q\ge 2.
\end{equation}

Let $B=B_{i,\la}$ for some fixed $i$. By \eqref{2.2}, we have $\sigma_\la e_\la=e_\la$. Consequently, in view of \eqref{2.12}, in order to prove Theorem~\ref{thmb},  it suffices to show the corresponding bounds for $\|B e_\la\|_{L^6(M)}$ with $\la\gg1$.

\subsection{Microlocal decomposition}\label{sec: microlocal decomposition}
We introduce a further microlocal decomposition, which
involves localizing in $\theta>\la^{-1/2}$ neighborhoods of geodesics in a fixed coordinate chart,
as in the recent work \cite{huang2024curvature}. This decomposition enables us to apply the bilinear estimate established in \cite{huang2024curvature}, whose proof builds on earlier bilinear oscillatory integral estimates developed in \cite{LeeBilinear} and \cite{TaoVargasVega}. Below we describe the  details of the decomposition used in \cite{huang2024curvature} in  dimension 2. 

First recall that the symbol $B(x,\xi)$ of $B$ in \eqref{2.8} is supported in a small
conic neighborhood of some $(x_0,\xi_0)\in S^*M$ with $p(x,\xi)\in (\la/2,2\la)$.  We may assume that its symbol has
small enough support so that we may work in a coordinate chart $\Omega$ for $M$, and that
\begin{align}\label{eq: assumption on local chart}
    x_0=0,\quad\xi_0=(0,1),\quad g_{jk}(0)=\delta^j_k
\end{align}
in the local coordinates. So we shall assume that $B(x,\xi)=0$ when $x$ is outside a small relatively compact neighborhood of the origin or $\xi$ is outside of a small conic neighborhood of $(0,1)$.

Next, let us define the microlocal cutoffs that we shall use.   We fix a function
$b\in C^\infty_0({\mathbb R}^{2})$ supported in $\{z=(z_1,z_2): \, |z_k|\le 1, \, \, 1\le k\le 2\}$
 which satisfies
\begin{equation}\label{m1}
\sum_{j\in {\mathbb Z}^{2}}b(z-j)\equiv 1.
\end{equation}
We shall use this function to build our microlocal cutoffs.
By the above, we shall focus on defining them 
 for $(y,\eta)\in S^*\Omega$ with    $y$ near the origin
 and  $\eta$ in a small conic neighborhood of $(0,1)$. 
We shall let
$$\Pi=\{y=(y_1,y_2): \, y_{2}=0\}$$
be the points in $\Omega$ whose last coordinate vanishes.

To construct the cutoffs associated to the $\theta$-net of geodesics in $S^*M$ that we require, let us first
set
\begin{equation}\label{map0}
b^\theta_j(y_1,\eta_1)=b(\theta^{-1}(y_1,\eta_1)-j)\in C^\infty_0({\mathbb R}^{2}),
\end{equation}
so that $\sum_{j\in {\mathbb Z}^{2}} b^\theta_j(y_1,\eta_1)=1$.  
Note that if $\Phi_t$ denote the geodesic flow on $S^*\Omega$,  the map
\begin{equation}\label{map1}
(t,x_1,\eta)\mapsto \Phi_t((x_1,0),\eta)\in S^*\Omega, \qquad ((x_1,0),\eta)\in S^*\Omega
\end{equation}
is a diffeomorphism from a neighborhood of $x_1=0$, $t=0$ and $(0,1)\in S^1$ to a neighborhood of $(0, (0,1))\in S^*\Omega$.

Next, write the inverse of \eqref{map1} as
$$S^*\Omega \ni (x,\omega)\to (\tau(x,\omega),\Psi(x,\omega), \Theta(x,\omega))\in \R \times \{y_1\in \R\}\times S^*_{(\Psi(x,\omega),0)}M.$$
Thus, the unit speed geodesic passing through $(x,\omega)\in S^*\Omega$ arrives at the hyperplane  $\Pi$ where $y_2=0$ at $(\Psi(x,\omega),0)\in \Pi$,
has covector $\Theta(x,\omega)\in S^*_{(\Psi(x,\omega),0)}\Omega$ there, and $\tau(x,\omega)=d_g(x,(\Psi(x,\omega),0))$ is the geodesic distance
between $x$ and the point $(\Psi(x,\omega),0)$ on this hyperplane.  We shall also let $ \Theta_1(x,\omega)$ denote the first coordinate
of the covector $\Theta(x,\omega)$, meaning that $\Theta(x,\omega)_1=\eta_1$ if $\Theta(x,\omega)=(\eta_1,\eta_2)\in S^*_{(\Psi(x,\omega),0)}\Omega$.

We can now define the microlocal cutoffs that we shall use.  For $(x,\xi)\in T^*\Omega\backslash0$ in a conic neighborhood of $(0,(0, 1))$, if
$b^\theta_j$ is as in \eqref{map0}, we define
\begin{equation}\label{b.45}
a^\theta_j(x,\xi)=b^\theta_j(\Psi(x,\xi/p(x,\xi)),  \Theta_1(x,\xi/p(x,\xi))),
\end{equation}
with $p(x,\xi)=\sqrt{g^{ij}(x)\xi_i\xi_j}$ being the principal symbol of $P=\sqrt{-\Delta_g}$.

Note that for $s$ near $0$ and $(x,\xi)\in S^*\Omega$ near $(0,(0,1))$
\begin{equation}\label{b.46}
a^\theta_j(\Phi_s(x,\xi))=a^\theta_j(x,\xi).
\end{equation}
Furthermore, 
 if $(y_1,\eta_1)=\theta j=\nu$ and $\gamma_\nu$ is the geodesic
in $S^*\Omega$ passing through $(y_1,0,\eta)\in S^*\Omega$ with 
$\eta\in S^*_{(y_1,0)}\Omega$ having $\eta_1$ as its first  coordinate
and $\eta_2>0$ then
\begin{equation}\label{b.48}
a^\theta_j(x,\xi)=0 \, \, 
\text{if } \, \text{dist}((x,\xi/p(x,\xi)),\gamma_\nu)\ge C_0\theta, \, \, \nu=\theta j
\end{equation}
for a uniform positive constant $C_0$. Also, $a_j^\theta$ satisfies the estimates
\begin{equation}\label{b.49}
|\partial_x^\sigma \partial_\xi^\gamma a^\theta_j(x,\xi)|\lesssim \theta^{-|\sigma|-|\gamma|},
\quad \text{if } \, \, p(x,\xi)=1,
\end{equation}
related to this support property.

Finally, if $\phi \in C^\infty_0(\Omega)$ equals one in a neighborhood of the $x$-support of $B(x,\xi)$, and if $\tilde \beta\in C^\infty_0((1/4,4))$ equals one in $(1/3,3)$, we define
\begin{equation}\label{qnusymbol}
A_\nu^\theta(x,\xi)=\phi(x) \, a_j^\theta(x,\xi) \, \tilde\beta\bigl(p(x,\xi)/\la\bigr),
\quad \nu =\theta j\in \theta \cdot {\mathbb Z}^{2}.
\end{equation}
By \eqref{map0} and \eqref{b.45}, it is not hard to check that the symbol $A^\theta_\nu(x,\xi)$ satisfies 
\begin{equation}\label{b.45a}
\bigl|\partial_x^\sigma \partial_\xi^\gamma A^\theta_\nu(x,\xi)\bigr| \lesssim \langle \xi \rangle^{\delta_0|\sigma|-(1-\delta_0)|\gamma|},
\end{equation}
 if $\theta\in [\la^{-\delta_0},1]$  for some $0<\delta_0<\frac{1}{2}$.
Hence the pseudo-differential operators
$A_\nu^\theta(x,D)$ with these symbols belong to a bounded subset
of
of $S^0_{1-\delta_0, \delta_0}(M)$. 

We shall need a few simple but very useful facts about these operators:
\begin{lemma}\label{alemma}  Let $\theta_0=\la^{-\delta_0}$ for some $0<\delta_0<\frac12$.  Then for any $h\in L^q(M)$
\begin{align}\label{2.33}
(\sum_\nu\|A^{\theta_0}_\nu h\|^q_{L^q(M)})^{\frac1q} &\lesssim \|h\|_{L^q(M)}, \quad 2\le q\le \infty.
\end{align}
Also, let \(A\) be a pseudodifferential operator whose symbol belongs to a bounded subset of \(S^0_{1-\delta_0,\delta_0}(M)\) and is supported in a small neighborhood of the support of \(B(x,\xi)\). Then, if \(\delta>0\) in \eqref{2.2} is sufficiently small,
\begin{equation}\label{commute}
\|A\sigma_\la A^{\theta_0}_\nu - AA^{\theta_0}_\nu \sigma_\la \|_{L^2\to L^6}=O(\la^{-\frac13+\frac43\delta_0}),
\end{equation}
\end{lemma}
In subsequent applications, we will take \(A = B\) or \(A = B\circ A^{\theta_1}_\mu\) for some \(\theta_1 \ge \theta_0\) and $\mu\in \theta_1\cdot\Z^2$, which means its symbol belongs to \(S^0_{1-\delta_0,\delta_0}(M)\).

\begin{proof}
To prove \eqref{2.33} we note that, by interpolation, it suffices to prove the inequality for $q=2$ and $q=\infty$.
The estimate for $q=2$ just follows from the fact that the $S^0_{1-\delta_0,\delta_0}$ operators $\{A^{\theta_0}_\nu\}$ are almost
orthogonal due to \eqref{m1}. Let $A^{\theta_0}_\nu (x,y)$ denote the Schwartz kernel of the operator $A^{\theta_0}_\nu$. The estimate for $q=\infty$ follows from the fact that the kernels satisfy
\begin{equation}\label{2.36}
\sup_x \int |A^{\theta_0}_\nu (x,y)|\, dy \le C.
\end{equation} 
To see \eqref{2.36}, we choose a local coordinate chart in which the second coordinate axis coincides with the unit-speed geodesic obtained by projecting $\gamma_\nu$ onto $M$ (for instance, Fermi normal coordinates). We refer the reader to Chapter 3.2 of \cite{SFIOII} for the invariance of pseudodifferential operators under changes of coordinates.
On the support of $A^{\theta_0}_\nu(x,\xi)$ we then have $|x_1|\lesssim \theta_0$ and $|\xi_1|/|\xi|\lesssim \theta_0$. Moreover,
$
a_j^\theta(x,\xi)\equiv 1
$
if $x_1=\xi_1=0$. Consequently, it is straightforward to check that
\begin{equation}\label{2.36a}
\bigl|\partial_{\xi_1}^{j}\partial_{\xi_2}^{k} A^{\theta_0}_\nu(x,\xi)\bigr|
\lesssim \la^{-(1-\delta_0)j-k}.
\end{equation}
Thus, a simple integration by parts argument then yields
\begin{equation}\label{ke}
A^{\theta_0}_\nu(x,y)
=
O_N\!\left(
\la^{2-\delta_0}
(1+\la^{1-\delta_0}|x_1-y_1|)^{-N}
(1+\la|x_2-y_2|)^{-N}
\right)
\end{equation}
for all $N>0$, which implies \eqref{2.36}.

To prove \eqref{commute}, let us
define the  wider cutoffs, after recalling \eqref{qnusymbol}, by setting
\begin{equation}\label{a}
\tilde A^{\theta_0}_\nu(x,\xi)=
\sum_{\{j\in {\mathbb Z}^{2}: \, |\theta_0j-\nu|\le C_0\theta_0\}}\phi(x) \, a_j^{\theta_0}(x,\xi) \, \tilde\beta\bigl(p(x,\xi)/\la\bigr).
\end{equation}
If we fix $C_0$ large enough, it is not hard to show that as operators between any $L^p(M)\to L^q(M)$ spaces for $1\le p,q\le \infty$, we have, 
\begin{equation}\label{c1}
 AA^{\theta_0}_\nu \sigma_\la= \tilde A^{\theta_0}_\nu AA^{\theta_0}_\nu \sigma_\la+O_N(\la^{-N}), \quad \forall \, N>0,
\end{equation}
as well as
\begin{equation}\label{c2}
 A \sigma_\la A^{\theta_0}_\nu= \tilde A^{\theta_0}_\nu A \sigma_\la A^{\theta_0}_\nu+O_N(\la^{-N}), \quad \forall \, N>0.
\end{equation}
\eqref{c1} follows from an integration by parts argument, together with the fact that, since the symbol of the operator \(A\) belongs to \(S^0_{1-\delta_0,\delta_0}(M)\), its Schwartz kernel \(A(x,y)\) is \(O(\lambda^{-N})\) when \(|x-y| \gtrsim \theta_0\).
To prove \eqref{c2}, 
 it suffices to show as operators between any $L^p(M)\to L^q(M)$ spaces for $1\le p,q\le \infty$
\begin{equation}\label{c2a}
  (I-\tilde A^{\theta_0}_\nu) A \sigma_\la A^{\theta_0}_\nu=O_N(\la^{-N}), \quad \forall \, N>0.
\end{equation}
Note that by \eqref{b.46}, if we choose $\delta>0$ in \eqref{2.2} sufficiently small, we may assume
\begin{equation}\label{b.46a}
a^\theta_j(\Phi_s(x,\xi))=a^\theta_j(x,\xi),\,\,\forall |s|\le 2\delta,\,\,\,\text{and}\,\,\,(x,\xi)\in\supp A(x,\xi).
\end{equation}
Thus, for any $(x,\xi)$ in the support of the operator $(I-\tilde A^{\theta_0}_\nu) A$, and 
$(y,\eta)$ in the support of the operator $ A^{\theta_0}_\nu$
we have 
\begin{equation}\label{b.46b}
\Phi_s(x,\xi))\neq (y,\eta) \,\,\,\ |s|\le 2\delta.
\end{equation}
Since, by \eqref{2.2a}, the operator $\sigma_\la$ involves an average of $e^{itP}$ for $|t|\le 2\delta$, \eqref{c2a} also follows from an integration by parts argument along with the fact that the singularities of 
$e^{itP}$
 propagate along the geodesic flow, see e.g., \cite[Theorem 4.3.5]{SoggeHangzhou}.

Additionally, let $A^{\theta_0}_1$ be the pseudodifferential operator with symbol
\begin{equation}\label{A1}
A^{\theta_0}_1(x,\xi)=\phi(x)\beta\Bigl(\frac{p(x,\xi)-\lambda}{C_0\lambda^{1-\delta_0}}\Bigr).
\end{equation}
By choosing $C_0$ sufficiently large, one can similarly show that, as operators from $L^p(M)$ to $L^q(M)$ for $1\le p,q\le \infty$,
\begin{equation}\label{c1a}
\tilde A^{\theta_0}_\nu A A^{\theta_0}_\nu \sigma\lambda
= A_1 \tilde A^{\theta_0}_\nu A A^{\theta_0}_\nu \sigma_\lambda + O_N(\lambda^{-N}), \quad \forall, N>0,
\end{equation}
and
\begin{equation}\label{c2aa}
\tilde A^{\theta_0}_\nu A \sigma\lambda A^{\theta_0}_\nu
= A_1 \tilde A^{\theta_0}_\nu A \sigma_\lambda A^{\theta_0}_\nu + O_N(\lambda^{-N}), \quad \forall, N>0.
\end{equation}
To see this, recall that, by \eqref{2.2a}, $\sigma_\lambda$ is an oscillatory integral operator with phase function $\lambda d_g(x,y)$. Since $p(x,\nabla_x d_g(x,y)) = 1$, both \eqref{c1a} and \eqref{c2aa} follow from a nonstationary phase argument via integration by parts.

Note that for fixed $x$, the $\xi$ support of the symbol of the operator $A^{\theta_0}_1\tilde A^{\theta_0}_\nu$ is in a contained in a ball of radius $\la^{1-\delta_0}$
 the variant of \eqref{ke} with $\theta_0$ replaced by $C\theta_0$. 
 Thus, a simple integration by parts argument yields
\begin{equation}\label{kea}
(A^{\theta_0}_1\tilde A^{\theta_0}_\nu )(x,y)
=
O_N\!\left(
\la^{2-2\delta_0}
(1+\la^{1-\delta_0}|x-y|)^{-N}
\right).
\end{equation}
By Young's inequality, \eqref{kea} implies
\begin{equation}\label{c3}
   \| A^{\theta_0}_1\tilde A^{\theta_0}_\nu h\|_{L^6(M)}\lesssim \la^{\frac13(2-2\delta_0)}\|h\|_{L^2(M)}.
\end{equation}
By \eqref{c1}, \eqref{c2} and \eqref{c3}, to prove \eqref{commute}, it remains to show that 
\begin{equation}\label{cc4}
\bigl\|
 A \sigma_\la A^{\theta_0}_\nu
-A A^{\theta_0}_\nu \sigma_\la 
\, 
\bigr\|_{L^2\to L^2}
=
O(\la^{-1+2\delta_0}).
\end{equation}

To prove this we recall that by  \eqref{2.2a}
$$\sigma_\la=(2\pi)^{-1}\int^{2\delta}_{-2\delta} \Hat \rho(t) e^{it\la} e^{-itP}\, dt.
$$
Therefore by
Minkowski's integral inequality,
we would have
\eqref{cc4} if 
\begin{equation}\label{cc5}
\sup_{|t|\le 2\delta}\, 
\bigl\| A e^{-itP} A^{\theta_0}_\nu
-A A^{\theta_0}_\nu e^{-itP} 
\, 
\bigr\|_{L^2\to L^2}
=
O(\la^{-1+2\delta_0}).
\end{equation}

Write
$$\bigl[ A e^{-itP} A^{\theta_0}_\nu
-A A^{\theta_0}_\nu e^{-itP} \bigr]
=A
\, \bigl[
(e^{-itP}A^{\theta_0}_\nu e^{itP}) -  A^{\theta_0}_\nu]
\circ e^{-itP}.
$$
Since $e^{-itP}$ also has $L^2$-operator norm one, we
would obtain \eqref{cc5} from
\begin{equation}\label{cc6}
\sup_{|t|\le 2\delta} \, 
\bigl\|
A
\, \bigl[
(e^{-itP}A^{\theta_0}_\nu e^{itP}) - A^{\theta_0}_\nu\bigr]
\, 
\bigr\|_{L^2\to L^2} 
=
O(\la^{-1+2\delta_0}).
\end{equation}
By Egorov's theorem for symbols belonging to a bounded subset
of $S^0_{1-\delta_0, \delta_0}(M)$ (see Taylor~\cite[P.147]{TaylorPDO}), we have the symbol of the operator $A
\, \bigl[
(e^{-itP}A^{\theta_0}_\nu e^{itP}) - A^{\theta_0}_\nu\bigl]$  belong to a bounded subset
of $S^{-1+2\delta_0}_{1-\delta_0, \delta_0}(M)$. Thus, \eqref{cc6} follows from the fact that the symbol of the operator is supported in $|\xi|_g\in [\la/4, 4\la]$. See also the proof of (2.40) in \cite{huang2024curvature} for further details.
\end{proof}

For convenience and by abuse of the notation, taking \(A = B\) or \(A = B\circ A^{\theta_1}_\mu\), we will write
\begin{align}\label{eq: defn of tilde sigma}
    \tilde \sigma_\la=A\circ \sigma_\la,
\end{align}
which is an operator depending on $A$. By \eqref{m1},  and \eqref{qnusymbol}, if $A=B$, we have that, as operators between any $L^p(M)\to L^q(M)$ spaces,
for $1\le p,q\le \infty$, and for $\theta\ge \la^{-\delta_0}$, 
\begin{equation}\label{m11}
\tilde \sigma_\la =\sum_\nu \tilde \sigma_\la A^{\theta_0}_\nu +O_N(\la^{-N}), \quad \forall \, N>0.
\end{equation}
This just follows from the fact that $R(x,D)=I-\sum_\nu A^\theta_\nu $ has symbol
supported outside of a $2\delta$ neighborhood of $B(x,\xi)$. Similarly, if \(A = B\circ A^{\theta_1}_\mu\) for some \(\theta_1 \ge \theta_0\) and $\mu\in \theta_1\cdot\Z^2$, we have 
\begin{equation}\label{m11a}
\tilde \sigma_\la =\sum_{|\nu-\mu|\le C_1\theta_1} \tilde \sigma_\la A^{\theta_0}_\nu +O_N(\la^{-N}), \quad \forall \, N>0,
\end{equation}
for some constant $C_1>0$, which can be chosen to be large. Moreover, as in \eqref{a}, if we define the wider cutoffs $\tilde A^{\theta}_\mu$ by setting
\begin{equation}\label{aw}
    \tilde A^{\theta_1}_\mu(x,\xi)=\sum_{\{j\in {\mathbb Z}^{2}: \, |\theta_1 j-\mu|\le C_0\theta_1\}}\phi(x) \, a_j^{\theta_1}(x,\xi) \, \tilde\beta\bigl(p(x,\xi)/\la\bigr),
\end{equation}
we also have
\begin{equation}\label{m11b}
\tilde \sigma_\la =\sum_{|\nu-\mu|\le C_1\theta_1} \tilde \sigma_\la A^{\theta_0}_\nu \tilde A^{\theta_1}_\mu +O_N(\la^{-N}), \quad \forall \, N>0.
\end{equation}
Both \eqref{m11a} and \eqref{m11b} follow from the same argument as in the proof of \eqref{c2}.
For later use, it is also straightforward to verify that the $\tilde A^{\theta_1}_\mu$ operator also satisfies the estimate \eqref{2.33}, in the same manner as $A^{\theta_1}_\mu$.

In view of \eqref{m11} we have for $\theta_0=\la^{-\delta_0}$
\begin{equation}\label{m13}
\bigl(\tilde \sigma_\la h\bigr)^2=\sum_{\nu, \nu'} \bigl(\tilde \sigma_\la A^{\theta_0}_{\nu} h\bigr) \,
\bigl(\tilde \sigma_\la A^{\theta_0}_{\nu'} h\bigr) \, + \, O(\la^{-N}\|h\|_{L^2(M)}^2).
\end{equation}
If
$\theta_0=\la^{-\delta_0}$ then the $\nu\in\theta_0\cdot  {\mathbb Z}^{2}$ index a $\la^{-\delta_0}$-separated set in
${\mathbb R}^{2}$.  We need to organize the pairs of indices $\nu, \nu'$ in \eqref{m13} as in many earlier works
(see \cite{LeeBilinear} and \cite{TaoVargasVega}).  We consider dyadic cubes $\tau^\theta_\mu$ in 
${\mathbb R}^{2}$ of side length $\theta=2^m\theta_0$, $m=0,1,\dots$, with
$\tau^\theta_\mu$ denoting translations of the cube $[0,\theta)^{2}$ by
$\mu\in\theta{\mathbb Z}^{2}$.  Then two such dyadic cubes of side length $\theta$ are said to be
{\em close} if they are not adjacent but have adjacent parents of side length $2\theta$, and, in that case, we write
$\tau^\theta_{\mu} \sim \tau^\theta_{\mu'}$.  Note that close cubes satisfy $\text{dist}(\tau^\theta_{\mu},\tau^\theta_{\mu'})
\in [\frac14\theta, 4\theta]$ and so each fixed cube has $O(1)$ cubes which are "close" to it.  Moreover, as noted in \cite{TaoVargasVega},
any distinct points $\nu,\nu'\in {\mathbb R}^{2}$ must lie in a unique pair of close cubes in this Whitney decomposition
of ${\mathbb R}^{2}$.  Consequently, there must be a unique triple $(\theta=\theta_02^m, \mu,\mu')$ such that
$(\nu,\nu')\in \tau^\theta_{\mu}\times \tau^\theta_{\mu'}$ and $\tau^\theta_{\mu}\sim \tau^\theta_{\mu'}$.  We remark that by choosing $B$
to have small support we need only consider $\theta=2^m\theta_0\ll 1$.

Taking these observations into account implies that the bilinear sum in \eqref{m13} can be organized as follows:
\begin{multline}\label{m14}
\sum_{\{k\in {\mathbb N}: \, k\ge 10 \, \, \text{and } \, 
\theta=2^k\theta_0\ll 1\}}
\sum_{\{(\mu, \mu'): \, \tau^\theta_\mu
\sim \tau^\theta_{ \mu'}\}}
\sum_{\{(\nu, \nu')\in
\tau^\theta_\mu\times \tau^\theta_{ \mu'}\}}
\bigl(\tilde \sigma_\la
A^{\theta_0}_\nu h\bigr) 
\cdot \bigl(\tilde \sigma_\la
A^{\theta_0}_{ \nu'} h\bigr)
\\
+\sum_{(\nu, \nu')\in \Xi_{\theta_0}} 
\bigl( \tilde \sigma_\la A^{\theta_0}_\nu h \bigr) 
\cdot \bigl( \tilde \sigma_\la
A^{\theta_0}_{ \nu'} 
h\bigr)
,
\end{multline}
where $\Xi_{\theta_0}$ indexes the remaining pairs such
that $|\nu- \nu'|\lesssim \theta_0=\la^{-\delta_0}$,
including the diagonal ones where $\nu= \nu'$.

Then the key estimate that we shall use in the proof of Theorem \ref{thmb}, which follows from variable coefficient bilinear harmonic analysis arguments, is the following:

\begin{proposition}\label{locprop}
Let $\theta_0=\la^{-\delta_0}$ for some $0<\delta_0<\frac12$, we have 
\begin{equation}\label{2.44}
\|B e_\la\|_{L^{6}(M)}\lesssim_\e \Bigl(\, \sum_\nu \|B A^{\theta_0}_\nu (e_\la)\|_{L^{6}(M)}^{6}\, 
\Bigr)^{\frac{1}{6}} + \la^{\frac{\delta_0}{6}+\e}\|e_\la\|^{\frac13}_{L^\infty(M)}\|e_\la\|^{\frac23}_{L^2(M)}+\la^{-\frac16+\frac89\delta_0}\|e_\la\|_{L^2}.
\end{equation}
 Similarly, if $\theta_1\geq\theta_0$ and $\tilde A^{\theta_1}_\mu$ is defined as in \eqref{aw}, we have, for some constant $C>0$,
 \begin{multline}\label{2.44a}
\|B A^{\theta_1}_\mu (e_\la)\|_{L^{6}(M)}\lesssim_\e \Bigl(\, \sum_{|\nu-\mu|\le C\theta_1} \|B A^{\theta_0}_\nu (e_\la)\|_{L^{6}(M)}^{6}\, 
\Bigr)^{\frac{1}{6}} \\+ \la^{\frac{\delta_0}{6}+\e}\|A^{\theta_1}_\mu e_\la\|^{\frac13}_{L^\infty(M)}\|\tilde A^{\theta_1}_\mu e_\la\|^{\frac23}_{L^2(M)}+\la^{-\frac16+\frac89\delta_0}\Bigl(\sum_{|\nu-\mu|\le C\theta_1}\|A^{\theta_0}_\nu e_\la\|_{L^2}^2\Bigr)^{\frac16}\|e_\la\|_{L^2}^{\frac23}.
\end{multline}
\end{proposition}

\begin{proof}
Recall our definition \eqref{eq: defn of tilde sigma} for $\tilde \sigma_\la=A\circ \sigma_\la$ and recall that by \eqref{m11}, we have
\begin{equation}\label{5.1}
\tilde \sigma_\la -\sum_\nu \tilde \sigma_\la A^{\theta_0}_\nu =R_\la
\quad \text{where } \, \, \|R_\la\|_{L^2\to L^\infty}=O_N(\la^{-N}) \quad \forall \, N>0.
\end{equation}
Thus, for any $h\in L^2(M)$,
\begin{equation}\label{5.2}
(\tilde \sigma_\la h)^2 =\sum_{\nu,\nu'}(\tilde \sigma_\la A^{\theta_0}_\nu h) \cdot (\tilde \sigma_\la A^{\theta_0}_{\nu'} h)
+O(\la^{-N}\|h\|^2_{L^2(M)}).
\end{equation}

Let us set, for any $h\in L^2(M)$,
\begin{equation}\label{diag}
\diag_A(h)=\sum_{(\nu,\nu')\in \Xi_{\theta_0}} (\tilde \sigma_\la A^{\theta_0}_\nu h)
\cdot (\tilde \sigma_\la A^{\theta_0}_{\nu'} h),
\end{equation}
and
\begin{equation}\label{5.4}
\far_A(h)=\sum_{(\nu,\nu')\notin \Xi_{\theta_0}} (\tilde \sigma_\la A^{\theta_0}_\nu h)
\cdot (\tilde \sigma_\la A^{\theta_0}_{\nu'} h) + O(\la^{-N}\|h\|_{L^2(M)}^2),
\end{equation}
with the last term containing the error terms in \eqref{5.2}.  Thus,
\begin{equation}\label{5.5}
(\tilde \sigma_\la h)^2 = \far_A(h)+\diag_A(h).
\end{equation}
Note that the summation in $\diag_A(h)$ is over near diagonal pairs $(\nu,\nu')$ by \eqref{m14}.  In particular,
for $(\nu,\nu')\in \Xi_{\theta_0}\subset \theta_0\cdot {\mathbb Z}^{2}\times \theta_0\cdot {\mathbb Z}^{2}$  we have $|\nu-\nu'|\le C\theta_0$ for 
some uniform positive constant $C$.  The other term $\far(h)$ in \eqref{5.5} includes the remaining pairs, many of which are
far from the diagonal, and this will contribute to the last term in \eqref{2.44}.

Let us further define
\begin{equation}\label{Tnu}
   T_\nu h=\sum_{\nu': \,(\nu, \nu')\in
\Xi_{\theta_0}}( \tilde\sigma_\la A^{\theta_0}_\nu h )(\tilde\sigma_\la A^{\theta_0}_{\nu'} h), 
\end{equation}
and write
\begin{equation}\label{b255}
\begin{aligned}
    ( \diag_A(h))^{2} &=\big(\sum_\nu T_\nu h\big)^2= \sum_{\nu_1, \nu_2}  T_{\nu_1} h \cdot T_{\nu_2} h.
    \end{aligned}
\end{equation}
As in \eqref{m14}, if we assume that $B(x,\xi)$ has small conic support, the 
sum in \eqref{b255} can be organized as
\begin{equation}\label{organize}
\begin{aligned}
    &\left( \sum_{\{k\in {\mathbb N}: \, k\ge 20 \, \, \text{and } \, 
\theta=2^k\theta_0\ll 1\}} \,  \, 
\sum_{\{(\mu_1, \mu_2): \, \tau^\theta_{\mu_1}
\sim \tau^\theta_{ \mu_2}\}}
\sum_{\{(\nu_1,\nu_2)\in
\tau^\theta_{\mu_1}\times \tau^\theta_{\mu_2}\}}
+\sum_{(\nu_1,  \nu_2)\in \tilde\Xi_{\theta_0}}
\right)T_{\nu_1} h T_{\nu_2} h,  \\
&={\tilde\Upsilon^{\text{far}}_A}(h)+{\tilde\Upsilon^{\text{diag}}_A}(h)
\end{aligned}
\end{equation}
Here $\tilde\Xi_{\theta_0}$ indexes the near diagonal pairs. This is another Whitney decomposition similar to \eqref{m14}, but the diagonal set $\tilde\Xi_{\theta_0}$ is much larger than the set $\Xi_{\theta_0}$ in \eqref{m14}. It is not hard to check that  $|\nu-\nu'|\le 2^{11}\theta_0$ if $(\nu,\nu')\in \Xi_{\theta_0}$
while  $|\nu_1-\nu_2|\le 2^{21}\theta_0$ if $(\nu_1,\nu_2)\in \tilde\Xi_{\theta_0}$.

We will need the following two lemmas.
\begin{lemma}\label{diag1}  If $A=B$, then
for any $h\in L^2(M)$, we have
\begin{equation}\label{star2}
\|{\tilde\Upsilon^{\text{diag}}_A}(h)\|_{L^{3/2}(M)}
\lesssim \left(\sum_\nu \|\tilde \sigma_\la
A^{\theta_0}_\nu h\|_{L^{6}(M)}^{6}\right)^{2/3}
+O(\la^{-N}\|h\|^4_{L^2(M)}).
\end{equation}
Similarly, if $A=B\circ A_\mu^{\theta_1}$, then
for any $h\in L^2(M)$, we have, for some constant $C>0$,
\begin{equation}\label{star2a}
\|{\tilde\Upsilon^{\text{diag}}_A}(h)\|_{L^{3/2}(M)}
\lesssim \left(\sum_{|\nu-\mu|\le C\theta_1}  \|\tilde \sigma_\la
A^{\theta_0}_\nu h\|_{L^{6}(M)}^{6}\right)^{2/3}
+O(\la^{-N}\|h\|^4_{L^2(M)}).
\end{equation}

\end{lemma}
\begin{proof}
We only give the proof of \eqref{star2}, the proof of \eqref{star2a} follows from the same argument. 
The idea is similar to the proof of Lemma A.1 in \cite{huang2024curvature}.
Let $\tilde A^{\theta_0}_\nu$ be defined as in \eqref{a}, recall that by \eqref{c2}, we have 
\begin{equation}\label{aa}
 \tilde\sigma_\la A^{\theta_0}_\nu= \tilde A^{\theta_0}_\nu \tilde \sigma_\la A^{\theta_0}_\nu+O_N(\la^{-N}), \quad  \forall\, N>0.
\end{equation}

Therefore, to prove \eqref{star2}, by duality,  it suffices to show that 
\begin{multline}\label{g'}
\left|\sum_{(\nu_1,  \nu_2)\in \tilde\Xi_{\theta_0}}\sum_{\nu_1',\nu'_2}
\int ( \tilde A^{\theta_0}_{\nu_1} \tilde \sigma_\la
 A^{\theta_0}_{\nu_1} h) (\tilde A^{\theta_0}_{\nu'_1} \tilde \sigma_\la
 A^{\theta_0}_{\nu'_1} h) ( \tilde A^{\theta_0}_{\nu_2} \tilde \sigma_\la
 A^{\theta_0}_{\nu_2} h) (\tilde A^{\theta_0}_{\nu'_2} \tilde \sigma_\la
 A^{\theta_0}_{\nu'_2} h) \cdot \overline{f} \,
dx \right|
\\
\lesssim\left(\sum_\nu \|\tilde \sigma_\la
A^{\theta_0}_\nu h\|_{L^6(M)}^{6}\right)^{2/3}
+O(\la^{-N}\|h\|_{L^2(M)}^4), \quad
\text{if } \, \, \|f\|_{L^3(M)}=1.
\end{multline}
Here $(\nu_1,  \nu'_1)\in \Xi_{\theta_0}$,  $(\nu_2,  \nu'_2)\in \Xi_{\theta_0}$, and the set $\tilde\Xi_{\theta_0}$ is as in \eqref{organize}.  So $\nu_1, \nu'_1, \nu_2, \nu'_2$ in \eqref{g'} satisfy $|\nu_1-\nu'_1|+|\nu_1-\nu_2|+|\nu_1-\nu'_2|=O(\la^{-\delta_0})$.

Let us also define the wider cutoffs $\tilde{\tilde A^{\theta_0}_\nu}$ by setting
\begin{equation}\label{aaa}
\tilde{\tilde A^{\theta_0}_\nu}(x,\xi)=
\sum_{\{j\in {\mathbb Z}^{2}: \, |\theta_0j-\nu|\le \tilde{C_0}\theta_0\}}\tilde \phi(x) \, a_j^\theta(x,\xi) \, \tilde{\tilde\beta}\bigl(p(x,\xi)/\la\bigr).
\end{equation}
Here $\tilde \phi(x)\in C_0^\infty(\Omega)$ equals one in a neighborhood of the support of $\phi$, and $\tilde{\tilde\beta}\in C_0^\infty(\R)$ equals one in a neighborhood of the support of $\tilde\beta(x)$. We choose $\tilde{C_0}$ such that $\tilde{C_0}\gg C_0$ where $C_0$ is the constant appearing in \eqref{a}. It is straightforward to verify that
the $\tilde{\tilde A^{\theta_0}_\nu}$ operator also satisfies the estimate \eqref{2.33}, in the same manner as $A^{\theta_0}_\nu$.

For later use, let us also recall that, by \eqref{a}, for each fixed $x$ the support of
$\xi\mapsto \tilde A^{\theta_0}_\nu(x,\xi)$ is contained in a
cone of aperture $\lesssim\theta_0= \la^{-\delta_0}$.  So if
$(\nu_1,  \nu_2)\in \tilde\Xi_{\theta_0}$  then all of  
$\xi \mapsto  \tilde A^{\theta_0}_{\nu_1}(x,\xi)$,
$\xi \mapsto  \tilde A^{\theta_0}_{\nu_1'}(x,\xi)$, $\xi \mapsto  \tilde A^{\theta_0}_{\nu_2}(x,\xi)$, $\xi \mapsto  \tilde A^{\theta_0}_{\nu_2'}(x,\xi)$ are supported
for every fixed $x$ in a common cone of aperture $O(\la^{-\delta_0})$. Thus,
if we fix $C_1$ sufficiently large and $\tilde {\tilde\beta}\in C_0^\infty(\R)$ equals one in sufficiently large neighborhood of the support of $\tilde\beta(x)$,  it is not hard to check that
\begin{multline}\label{c}
\text{if } \, (\nu_1,  \nu_2)\in \tilde\Xi_{\theta_0}  \, \,
\text{and } \, \, \bigl(1- {\tilde{\tilde A^{\theta_0}_{\nu_1}}(y,\xi)} \bigr)
A^{\theta_0}_{\nu_1}(y,\xi_1)A^{\theta_0}_{\nu'_1}(y,\xi_2)A^{\theta_0}_{\nu_2}(y,\xi_3)A^{\theta_0}_{\nu'_2}(y,\xi_4)\ne0,
\\
\text{then } \, \, |\xi-(\xi_1+\xi_2+\xi_3+\xi_4)|\ge c\theta_0\la,
\end{multline}
for some fixed constant $c>0$. 
Therefore,  by an integration by parts argument we have
 \begin{equation*}
     \|(I-\tilde{\tilde A^{\theta_0}_{\nu_1}})^*
\bigl(\tilde A^{\theta_0}_{\nu_1} h_1 \cdot \tilde A^{\theta_0}_{\nu'_1} h_2 \cdot \tilde A^{\theta_0}_{\nu_2} h_3 \cdot \tilde A^{\theta_0}_{\nu'_2} h_4\bigr)
\|_{L^\infty(M)}
\lesssim_N\la^{-N}\prod_{i=1}^4\|h_i\|_{L^1(M)}  \qquad
\forall \, N>0.
 \end{equation*}
Thus, modulo $O(\la^{-N}\|h\|_{L^2(M)}^4)$ errors, the left side of \eqref{g'} is dominated by

\begin{equation}\label{gg'}
    \begin{aligned}
   \Bigl|&\sum_{(\nu_1,  \nu_2)\in \tilde\Xi_{\theta_0}} \sum_{\nu_1',\nu'_2}
\int ( \tilde A^{\theta_0}_{\nu_1} \tilde \sigma_\la
 A^{\theta_0}_{\nu_1} h) (\tilde A^{\theta_0}_{\nu'_1} \tilde \sigma_\la
 A^{\theta_0}_{\nu'_1} h) ( \tilde A^{\theta_0}_{\nu_2} \tilde \sigma_\la
 A^{\theta_0}_{\nu_2} h) (\tilde A^{\theta_0}_{\nu'_2} \tilde \sigma_\la
 A^{\theta_0}_{\nu'_2} h) \cdot \overline{f} \,
dx \Bigr|
\\
&\le \bigl(\sum_{(\nu_1,  \nu_2)\in \tilde\Xi_{\theta_0}}\sum_{\nu_1',\nu'_2}
\|  ( \tilde A^{\theta_0}_{\nu_1} \tilde \sigma_\la
 A^{\theta_0}_{\nu_1} h) (\tilde A^{\theta_0}_{\nu'_1} \tilde \sigma_\la
 A^{\theta_0}_{\nu'_1} h) ( \tilde A^{\theta_0}_{\nu_2} \tilde \sigma_\la
 A^{\theta_0}_{\nu_2} h) (\tilde A^{\theta_0}_{\nu'_2} \tilde \sigma_\la
 A^{\theta_0}_{\nu'_2} h)\|_{L^{3/2}(M)}^{3/2}
\bigr)^{2/3} \\
&\qquad
\cdot \bigl(\sum_{(\nu_1,  \nu_2)\in \tilde\Xi_{\theta_0}}
\|\tilde{\tilde A^{\theta_0}_{\nu_1}} f\|_{L^3(M)}^3\bigr)^{1/3}
\\
&\lesssim 
\bigl(\sum_\nu \|\tilde A^{\theta_0}_\nu 
\tilde \sigma_\la  A^{\theta_0}_\nu h\|^{6}_{L^6(M)}
\bigr)^{2/3} \cdot \bigl(
\sum_\nu\| \tilde{\tilde A^{\theta_0}_\nu} f\|_{L^3(M)}^3\bigr)^{1/3}
\\
&\lesssim \bigl(\sum_\nu \|\tilde A^{\theta_0}_\nu 
\tilde \sigma_\la A^{\theta_0}_\nu h\|^{6}_{L^6(M)}
\bigr)^{2/3},  
    \end{aligned}
\end{equation}
using H\"older's inequality, the fact that
if $\nu_1$ is fixed there are just $O(1)$ indices $\nu'_1, \nu_2$ and $\nu'_2$ with
$(\nu_1,\nu_2)\in \tilde\Xi_{\theta_0}$, $(\nu_1,\nu'_1)\in \Xi_{\theta_0}$ and $(\nu_2,\nu'_2)\in \Xi_{\theta_0}$, followed by
\eqref{2.33} with $q=3$ for the $\tilde{\tilde A^{\theta_0}_\nu}$ operator.  Based on this, modulo
$O(\la^{-N}\|h\|_{L^2(M)}^4)$, the left side
of \eqref{g'} is dominated by
$(\sum_\nu \|\tilde A^{\theta_0}_\nu \tilde \sigma_\la
 A^{\theta_0}_\nu h\|_{L^6}^{6})^{2/3}$.
By using \eqref{c2} again we conclude that this last expression is dominated by $(\sum_\nu \|\tilde \sigma_\la A^{\theta_0}_\nu h\|_{L^6(M)}^{6})^{2/3}+ O(\la^{-N}\|h\|_{L^2(M)}^4)$, which yields \eqref{star2} and completes the proof of Lemma~\ref{diag1}.
\end{proof}

\begin{lemma}\label{leelemma} Let $\far_A$ be as in \eqref{5.4}, and, as above $\theta_0=\la^{-\delta_0}$. Then for all $\e>0$ we have
\begin{equation}\label{5.7}
\int_M |\far_A(h)|^{2}\, dx \lesssim_\e\la^{\delta_0+\e} \, 
\|h\|^4_{L^2(M)},
\end{equation}
assuming, as in Proposition~\ref{locprop}, that the conic support of $B(x,\xi)$ in \eqref{2.8} as well as $\delta$ and $\tilde\delta$ in \eqref{2.2} are sufficiently small.
Similarly, if  ${\tilde\Upsilon^{\text{far}}_A}(h)$ is as in \eqref{organize}, we have
\begin{equation}\label{5.7'}
\int_M |{\tilde\Upsilon^{\text{far}}_A}(h)|\, dx \lesssim_\e\la^{\delta_0+\e} \, 
\|h\|^4_{L^2(M)}.
\end{equation}
\end{lemma}
\begin{proof}
   \eqref{5.7} and \eqref{5.7'} follow from Lemma A.2 in \cite{huang2024curvature}. Although Lemma A.2 is proved there for the special case $\delta_0=\frac18$ and $A=B$, the same argument applies more generally for any $0<\delta_0<1/2$ and for $A=B\circ A_\mu^{\theta_1}$. The proof uses bilinear oscillatory integral estimates of Lee~\cite{LeeBilinear} and arguments of previous work of Blair and Sogge in \cite{BlairSoggeRefined,blair2015refined,SBLog}. And this is also where we used the condition on $\delta, \tilde\delta$ in \eqref{2.2a}, in order to apply the bilinear oscillatory integral theorems of Lee ~\cite{LeeBilinear}. See the appendix in \cite{huang2024curvature} for more details.\end{proof}

   Note that if $A=B\circ A_\mu^{\theta_1}$, it is not hard to see 
   \eqref{m11b} implies that $\far_A(h)=\far_A(\tilde A^{\theta_1}_\mu h)+O(\la^{-N})$ and ${\tilde\Upsilon^{\text{far}}_A}(h)={\tilde\Upsilon^{\text{far}}_A}(\tilde A^{\theta_1}_\mu h)+O(\la^{-N})$. Thus,  \eqref{5.7} and \eqref{5.7'} also implies that
\begin{equation}\label{5.7a}
\int_M |\far_A(h)|^{2}\, dx \lesssim_\e\la^{\delta_0+\e} \, 
\|\tilde A^{\theta_1}_\mu h\|^4_{L^2(M)},
\end{equation}
as well as 
\begin{equation}\label{5.7b}
\int_M |{\tilde\Upsilon^{\text{far}}_A}(h)|\, dx \lesssim_\e\la^{\delta_0+\e} \, 
\|\tilde A^{\theta_1}_\mu h\|^4_{L^2(M)}.
\end{equation}

Now we continue the proof of Proposition~\ref{locprop} and we will use $A$ to represent either $B$ or $B \circ A^{\theta_1}_\mu $. 
Note that 
$$| \tilde \sigma_\la e_\la \,  \tilde \sigma_\la e_\la |^{3}
\le 2 | \tilde \sigma_\la e_\la \,  \tilde \sigma_\la e_\la |
\cdot
\bigl(|\diag_A(e_\la)|^{2}+|\far_A(e_\la)|^{2}\bigr).
$$
So we have 
\begin{align}\label{5.8}
\begin{split}
    &\|A e_\la\|^{6}_{L^{6}(M)}=\|\tilde \sigma_\la e_\la\|^{6}_{L^{6}(M)}=\int_{M}  | \tilde \sigma_\la e_\la \cdot \tilde \sigma_\la e_\la|^{3} \, dx\\
    &\le  2\int_{M} | \tilde \sigma_\la e_\la \cdot \tilde \sigma_\la e_\la| \, |\diag_A(e_\la)|^{2} \, dx+ 2\int_{M} | \tilde \sigma_\la e_\la \cdot \tilde \sigma_\la e_\la| \, |\far_A(e_\la)|^{2} \, dx  = I + II.
\end{split}
\end{align}

To estimate $II$, note that $\tilde\sigma_\la e_\la=A e_\la$. If $A=B$,
by \eqref{2.9} and  \eqref{5.7}, we have
\begin{equation*}
II\lesssim_\e \|\tilde \sigma_\la e_\la\|^{2}_{L^\infty(M)} \cdot \la^{\delta_0+\e} \|e_\la\|^4_{L^2(M)}
\lesssim \la^{\delta_0+\e}\|e_\la\|^2_{L^\infty(M)}\|e_\la\|^4_{L^2(M)}.
\end{equation*}
Hence $II^{1/6}$ is dominated by the second term in \eqref{2.44}.
 If $A=B\circ A_\mu^{\theta_1}$,
by \eqref{2.9} and  \eqref{5.7a}, we have
\begin{equation*}
II\lesssim_\e \|\tilde \sigma_\la e_\la\|^{2}_{L^\infty(M)} \cdot \la^{\delta_0+\e} \|\tilde A^{\theta_1}_\mu e_\la\|^4_{L^2(M)}
\lesssim \la^{\delta_0+\e}\| A^{\theta_1}_\mu e_\la\|^2_{L^\infty(M)}\|\tilde A^{\theta_1}_\mu e_\la\|^4_{L^2(M)}.
\end{equation*}
Hence $II^{1/6}$ is dominated by the second term in \eqref{2.44a}.

To estimate the first term $I$,  by \eqref{organize}, we have 
\begin{equation}\label{5.8n2}
\begin{aligned}
    I&=
2\int_{M} | \tilde \sigma_\la e_\la \cdot \tilde \sigma_\la e_\la| \, |\diag_A(e_\la)|^{2} \, dx  \\
&\le 2\int_{M} | \tilde \sigma_\la e_\la \cdot \tilde \sigma_\la e_\la| \, |{\tilde\Upsilon^{\text{diag}}_A}(e_\la)| \, dx +2\int_{M} | \tilde \sigma_\la e_\la \cdot \tilde \sigma_\la e_\la| \, |{\tilde\Upsilon^{\text{far}}_A}(e_\la)| \, dx \\
&= \bA+\bB 
\end{aligned}
\end{equation}

To estimate $\bB$, if $A=B$, we use \eqref{5.7'} to get
\begin{equation*}
    \bB\lesssim_\e \|\tilde \sigma_\la e_\la\|^{2}_{L^\infty(M)} \cdot \la^{\delta_0+\e}  \|e_\la\|^4_{L^2(M)}
\lesssim_\e  \la^{\delta_0+\e}\|e_\la\|^2_{L^\infty(M)}\|e_\la\|^4_{L^2(M)},
\end{equation*}
so $\bB^{1/6}$ is also dominated by the second term in \eqref{2.44}.
If $A=B\circ A_\mu^{\theta_1}$,  by \eqref{5.7b}, we have
\begin{equation*}
    \bB\lesssim_\e \|\tilde \sigma_\la e_\la\|^{2}_{L^\infty(M)} \cdot \la^{\delta_0+\e}  \|\tilde A^{\theta_1}_\mu e_\la\|^4_{L^2(M)}
\lesssim_\e  \la^{\delta_0+\e}\| A^{\theta_1}_\mu e_\la\|^2_{L^\infty(M)}\|\tilde A^{\theta_1}_\mu e_\la\|^4_{L^2(M)},
\end{equation*}
so $\bB^{1/6}$ is dominated by the second term in \eqref{2.44a}.

Thus, we just need to see that we also have suitable bounds for $\bA^{1/6}$. By  H\"older's inequality, Young's inequality, \eqref{star2} and \eqref{star2a}, we have
\begin{align*}
\bA
&\le 2\|\tilde\sigma_\la e_\la \,  \tilde\sigma_\la e_\la\|_{L^{3}(M)} \cdot \|{\tilde\Upsilon^{\text{diag}}_A}(e_\la)\|_{L^{3/2}(M)}
\\
&\le \tfrac{2}{3}  \|\tilde\sigma_\la e_\la \,  \tilde\sigma_\la e_\la\|_{L^{3}(M)}^{3} + \tfrac{4}{3} \|{\tilde\Upsilon^{\text{diag}}_A}(e_\la)\|_{L^{3/2}(M)}^{3/2}
\\
&\le \begin{cases}
     \tfrac{2}{3}  \|\tilde\sigma_\la e_\la \,  \tilde\sigma_\la e_\la \|_{L^{3}(M)}^{3} +C
 \sum_{\nu}  \, \|\tilde\sigma_\la A^{\theta_0}_{\nu}  e_\la\|_{L^{6}(M)}^{6}\, \, + \, O(\la^{-N}\|e_\la\|^6_{L^2(M)}),\,\,\,\text{if}\,\,A=B \\
  \tfrac{2}{3}  \|\tilde\sigma_\la e_\la \,  \tilde\sigma_\la e_\la \|_{L^{3}(M)}^{3} +C
\sum_{|\nu-\mu|\le C\theta_1}  \, \|\tilde\sigma_\la A^{\theta_0}_{\nu}  e_\la\|_{L^{6}(M)}^{6}\, \, + \, O(\la^{-N}\|e_\la\|^6_{L^2(M)}),\,\,\,\text{if}\,\,A=B\circ A_\mu^{\theta_1}
\end{cases}
\end{align*}
for some positive constant $C>0$.
The first term on the right can be absorbed in the left side of \eqref{5.8}. It suffices to estimate the second term.

If $A=B$, we have
\begin{equation}\label{260}
    \begin{aligned}
      &\sum_\nu \|\tilde \sigma_\la A^{\theta_0}_\nu e_\la\|^{6}_{L^{6}(M)}=\sum_\nu \|\tilde \sigma_\la A^{\theta_0}_\nu e_\la\|_{L^6}^2\cdot \|\tilde \sigma_\la A^{\theta_0}_\nu e_\la\|_{L^6}^{4}\\
    \lesssim&\sum_\nu \|\tilde \sigma_\la A^{\theta_0}_\nu e_\la\|_{L^6}^2\cdot \| BA^{\theta_0}_\nu\sigma_\la  e_\la\|_{L^6}^{4}+ \sum_\nu \|\tilde \sigma_\la A^{\theta_0}_\nu e_\la\|_{L^6}^2\cdot \| (B\sigma_\la A^{\theta_0}_\nu-BA^{\theta_0}_\nu\sigma_\la) e_\la\|_{L^6}^{4}\\
    \lesssim& \sum_\nu \|\tilde \sigma_\la A^{\theta_0}_\nu e_\la\|_{L^6}^2\cdot \| B A^{\theta_0}_\nu\sigma_\la e_\la \|_{L^6}^{4}+\sum_\nu \la^{\frac1{3}}\|A^{\theta}_\nu e_\la\|_{L^2}^2\cdot \la^{4(-\frac13+\frac43\delta_0)}\|e_\la\|_{L^2}^{4}\\
    \lesssim & \sum_\nu \|\tilde \sigma_\la A^{\theta_0}_\nu e_\la \|_{L^6}^2 \cdot \| B A^{\theta_0}_\nu\sigma_\la  e_\la \|_{L^6}^{4}+\la^{\frac13+4(-\frac13+\frac43\delta_0)}\|e_\la\|_{L^2}^{6}.   
    \end{aligned}
\end{equation}
In the second inequality we used \eqref{commute} and the fact that 
\begin{equation}\label{local1}
\|\tilde \sigma_\la\|_{L^2(M) \to L^6(M)}\lesssim \la^{\frac16},
\end{equation}
which is a consequence of \eqref{2.9} and the $L^q$ spectral cluster bounds of Sogge \cite{sogge881}. In the third inequality, we used \eqref{2.33}.
Using H\"older's inequality in the last line, we obtain
\begin{align}\label{2}
    \sum_\nu \|\tilde \sigma_\la A^{\theta_0}_\nu e_\la\|^{6}_{L^{6}(M)}\le C\Bigl(\sum_\nu \|\tilde \sigma_\la A^{\theta_0}_\nu e_\la\|_{L^6}^{6}\Bigr)^{\frac13}\Bigl(\sum_\nu \|B A^{\theta_0}_\nu \sigma_\la e_\la\|_{L^6}^{6}\Bigr)^{\frac{2}{3}} +C\la^{-1+\frac{16}{3}\delta_0}\|e_\la\|_{L^2}^{6},
\end{align}
with some constant $C>0$.
By Young's inequality, for any $D>0$, we can bound the right-hand side above as follows
\begin{align*}
    &C\Bigl(\sum_\nu \|\tilde \sigma_\la A^{\theta_0}_\nu e_\la \|_{L^6}^{6}\Bigr)^{\frac1{3}}\Bigl(\sum_\nu \|B A^{\theta_0}_\nu \sigma_\la e_\la\|_{L^6}^{6}\Bigr)^{\frac{2}{3}}\\
    =&C D \Bigl(\sum_\nu \|\tilde \sigma_\la A^{\theta_0}_\nu e_\la\|_{L^6}^{6}\Bigr)^{\frac1{3}}\cdot D^{-1}\Bigl(\sum_\nu \|B A^{\theta_0}_\nu \sigma_\la e_\la\|_{L^6}^{6}\Bigr)^{\frac{2}{3}}\\
    \le& C \left( \tfrac1{3}D^{3} \sum_\nu \|\tilde \sigma_\la A^{\theta_0}_\nu e_\la\|_{L^6}^{6}+ \tfrac{2}{3}D^{-\frac{3}{2}}\sum_\nu \| BA^{\theta_0}_\nu \sigma_\la e_\la\|_{L^6}^{6} \right).
\end{align*}
If $D>0$ is chosen to be small enough so that $\tfrac13CD^{3}<\frac{1}{2}$, we can absorb the contribution of the first term above into the left-hand side of \eqref{2}, and conclude that
\begin{align}\label{3}
    \sum_\nu \|\tilde \sigma_\la A^{\theta_0}_\nu e_\la\|^{6}_{L^{6}(M)}&\lesssim \sum_\nu \|B A^{\theta_0}_\nu \sigma_\la e_\la \|_{L^6}^{6}+\la^{-1+\frac{16}{3}\delta_0}\|e_\la\|_{L^2}^{6}
    \lesssim  \sum_\nu \| BA^{\theta_0}_\nu e_\la \|_{L^6}^{6}+\la^{-1+\frac{16}{3}\delta_0}\|e_\la\|_{L^2}^{6}. 
\end{align}
This finishes the proof of \eqref{2.44}.

Similarly, if $A=B\circ A_\mu^{\theta_1} $
 we have
\begin{align*}
    &\sum_{|\nu-\mu|\le C\theta_1} \|\tilde \sigma_\la A^{\theta_0}_\nu e_\la\|^{6}_{L^{6}(M)}=\sum_{|\nu-\mu|\le C\theta_1}\|\tilde \sigma_\la A^{\theta_0}_\nu e_\la\|_{L^6}^2\cdot \|\tilde \sigma_\la A^{\theta_0}_\nu e_\la\|_{L^6}^{4}\\
    \lesssim&\sum_{|\nu-\mu|\le C\theta_1} \|\tilde \sigma_\la A^{\theta_0}_\nu e_\la\|_{L^6}^2\cdot \|A_\mu^{\theta_1} B  A^{\theta_0}_\nu\sigma_\la  e_\la\|_{L^6}^{4}+\sum_\nu \|\tilde \sigma_\la A^{\theta_0}_\nu e_\la\|_{L^6}^2\cdot \| (A_\mu^{\theta_1}B-B A_\mu^{\theta_1})  A^{\theta_0}_\nu\sigma_\la e_\la\|_{L^6}^{4}\\
    &+\sum_{|\nu-\mu|\le C\theta_1} \|\tilde \sigma_\la A^{\theta_0}_\nu e_\la\|_{L^6}^2\cdot \| (B A_\mu^{\theta_1} \sigma_\la A^{\theta_0}_\nu-B A_\mu^{\theta_1} A^{\theta_0}_\nu\sigma_\la) e_\la\|_{L^6}^{4}\\
    \lesssim& \sum_{|\nu-\mu|\le C\theta_1} \|\tilde \sigma_\la A^{\theta_0}_\nu e_\la\|_{L^6}^2\cdot \|   A_\mu^{\theta_1}B A^{\theta_0}_\nu\sigma_\la e_\la \|_{L^6}^{4}\\
    &\quad+\sum_{|\nu-\mu|\le C\theta_1} \la^{\frac1{3}}\|A^{\theta}_\nu e_\la\|_{L^2}^2\cdot (\la^{4(-\frac13+\frac43\delta_0)}+\la^{4(\frac16 - 1 + 2\delta_0)})\|e_\la\|_{L^2}^{4}\\
    \lesssim & \sum_{|\nu-\mu|\le C\theta_1} \|\tilde \sigma_\la A^{\theta_0}_\nu e_\la \|_{L^6}^2 \cdot \| B A^{\theta_0}_\nu\sigma_\la  e_\la \|_{L^6}^{4}+\la^{\frac13+4(-\frac13+\frac43\delta_0)}\Bigl(\sum_{|\nu-\mu|\le C\theta_1}\|A^{\theta}_\nu e_\la\|_{L^2}^2\Bigr)\|e_\la\|_{L^2}^{4},
\end{align*}
if $0<\delta_0<\tfrac12$.
In the second inequality, we used \eqref{commute} and the estimate
\begin{equation}\label{comu}
    \| (A_\mu^{\theta_1}B - B A_\mu^{\theta_1}) A^{\theta_0}_\nu \sigma_\la e_\la \|_{L^6}
    \lesssim \la^{\frac16 - 1 + 2\delta_0} \|e_\la\|_{L^2}.
\end{equation}
To justify this, note that for \(\theta_1 \ge \la^{-\delta_0}\), the symbol of the operator 
\((A_\mu^{\theta_1}B - B A_\mu^{\theta_1})\) is supported in the region 
\(p(x,\xi) \in [\la/2, 2\la]\) and belongs to a bounded subset of 
\(S^{-1+2\delta_0}_{1-\delta_0, \delta_0}(M)\). 
Consequently, one has
\[
\|(A_\mu^{\theta_1}B - B A_\mu^{\theta_1})\|_{L^6 \to L^6}
\lesssim \la^{-1 + 2\delta_0}.
\]
The bound \eqref{comu} then follows from \eqref{2.33} together with Sogge’s \(L^q\) spectral cluster estimates \cite{sogge881}.

If we repeat the arguments in \eqref{2}-\eqref{3}, we conclude that 
\begin{align}\label{3a}
    \sum_{|\nu-\mu|\le C\theta_1}  \|\tilde \sigma_\la A^{\theta_0}_\nu e_\la\|^{6}_{L^{6}(M)}
    \lesssim  \sum_{|\nu-\mu|\le C\theta_1} \| BA^{\theta_0}_\nu e_\la \|_{L^6}^{6}+\la^{-1+\frac{16}3\delta_0}(\sum_{|\nu-\mu|\le C\theta_1}\|A^{\theta}_\nu e_\la\|_{L^2}^2)\|e_\la\|_{L^2}^{4}.
\end{align}
This finishes the proof of \eqref{2.44a}.
\end{proof}

\subsection{Proof of Theorem \ref{thma1}}

The main step is to upgrade Proposition \ref{locprop} via iteration as follows.

\begin{theorem}\label{thmb}
    Suppose that there exists $0<\delta_\infty<\eta<\frac{1}{2}$ so that for all $\theta_0=\la^{-\delta_0}$ with $0<\delta_0<\frac12$ we have
    \begin{equation}\label{n3}
        \sup_\nu\|A_\nu^{\theta_0} e_\la\|_{L^\infty(M)}\lesssim_\e \la^{(\frac{1}{2}-\delta_\infty)(1-\delta_0)+\e}, 
    \end{equation}
  \begin{equation}\label{i3a}
        \sup_\nu\|A_\nu^{\theta_0} e_\la\|_{L^2(X)}\lesssim_\e \la^{-\eta\delta_0+\e},
    \end{equation}   as well as 
 \begin{equation}\label{n3a}
      \|e_\la\|_{L^\infty(M)}\lesssim_\e \la^{\frac{1}{2}-\delta_\infty+\e}.
    \end{equation}   
 Then we have
 \begin{equation}\label{n2}
       \|Be_\la\|_{L^6(M)}\lesssim \Bigl(\sum_\nu \| B A^{\theta_0}_\nu e_\la \|_{L^6(M)}^6\Bigr)^{\frac16}+\la^{\frac{1}{3}(\frac{1}{2}-\delta_\infty)+\e}
    \end{equation}
   holds for any $\theta_0=\la^{-\delta_0}$ with $0<\delta_0\leq\frac{3}{8}(1-{\delta_\infty})$, where the implied constant depends on $\e, \delta_0,\delta_\infty,\eta$ and the implied constants appearing in \eqref{n3}--\eqref{n3a}.
\end{theorem}

\begin{proof}
    For any fixed $\epsilon_0>0$, we want to show that
     \begin{equation}\label{n2'}
       \|Be_\la\|_{L^6(M)}\lesssim \Bigl(\sum_\nu \| B A^{\theta_0}_\nu e_\la \|_{L^6(M)}^6\Bigr)^{\frac16}+\la^{\frac{1}{3}(\frac{1}{2}-\delta_\infty)+\e_0}.
    \end{equation}
    We will prove by induction that \eqref{n2'} holds for $\theta_0=\la^{-\delta_0}$ with
   \begin{align*}
       0<\delta_0\leq \min\left\{\e_0\sum_{i=0}^n \left(1+2(\eta-\delta_\infty)\right)^i,\; \tfrac{3}8(1-{\delta_\infty})\right\}
   \end{align*}
   for any integer $n\geq0$. The implied constant in \eqref{n2'} wouldn't depend on $n$ because our assumption $0<\delta_\infty<\eta<\frac{1}{2}$ ensures that $1+2(\eta-\delta_\infty)>1$, so the induction will stop in finitely many steps.

    By \eqref{2.44} and \eqref{n3a}, we have 
\begin{equation*}
      \|B e_\la\|_{L^{6}(M)}\lesssim_\e \Bigl(\, \sum_\nu \|BA^{\theta_0}_\nu (e_\la)\|_{L^{6}(M)}^{6}\Bigr)^{1/6} + \la^{\frac{\delta_0}{6}+\frac{1}{3}(\frac{1}{2}-\delta_\infty)+\e}+\la^{-\frac16+\frac89\delta_0} .
\end{equation*}
Hence,  we have \eqref{n2'} holds with $0<\delta_0\le \e_0$, as long as we choose $\e<\frac12\e_0$, which is the base case $n=0$. Note that our assumption on the range for $\delta_0$ guarantees that the last term above is always dominated by the second term on the right side of \eqref{n2}.

Now we prove the induction step. Assume we have the induction hypothesis: for $\theta_1=\la^{-\delta_1}$, with
\[
0<\delta_1\le\e_0\sum_{i=0}^n \left(1+2(\eta-\delta_\infty)\right)^i\leq \tfrac{3}{8}(1-{\delta_\infty}),
\]
we have
\begin{equation*}
       \|Be_\la\|_{L^6(M)}\lesssim  \Bigl(\sum_\mu \| B A^{\theta_1}_\mu e_\la \|_{L^6(M)}^6\Bigr)^{\frac16}+\la^{\frac{1}{3}(\frac{1}{2}-\delta_\infty)+\e_0}.
\end{equation*}
By applying \eqref{2.44a} with $\theta_0=\la^{-\delta_0}\leq \theta_1=\la^{-\delta_1}$ to the first term on the right side above, we have 
\begin{equation}\label{2.44c}
\begin{aligned}
    \Bigl(\, \sum_\mu \|BA^{\theta_1}_\mu e_\la\|_{L^{6}(M)}^{6}\Bigr)^{1/6}&\lesssim_\e  \Bigl(\, \sum_\mu \sum_{|\nu-\mu|\le C\theta_1} \|B A^{\theta_0}_\nu e_\la\|_{L^{6}(M)}^{6}\Bigr)^{1/6}\\
    &\quad +  \la^{\frac{\delta_0}{6}+\e} \Bigl(\, \sum_\mu \|A^{\theta_1}_\mu e_\la\|^{2}_{L^\infty(M)}\|\tilde A^{\theta_1}_\mu e_\la\|^{4}_{L^2(M)}\Bigr)^{1/6} \\
    &\quad +  \la^{-\frac16+\frac89\delta_0} \Bigl(\, \sum_\mu \sum_{|\nu-\mu|\le C\theta_1}\|A^{\theta_0}_\nu e_\la\|_{L^2}^2\|e_\la\|_{L^2}^{4}\Bigr)^{1/6}.
\end{aligned}
\end{equation}
Recall that  $\mu\in\theta_1\cdot  {\mathbb Z}^{2}$ 
 index a $\la^{-\delta_1}$-separated set in
${\mathbb R}^{2}$, we have for each fixed $\nu$,  $|\{\mu: |\nu-\mu|\le C\theta_1\}|=O(1)$. Hence the first term on the right side is bounded by
\begin{equation}\label{2.44d}
\begin{aligned}
  \Bigl(\, \sum_\mu \sum_{|\nu-\mu|\le C\theta_1} \|B A^{\theta_0}_\nu e_\la\|_{L^{6}(M)}^{6}\Bigr)^{1/6}\lesssim  \Bigl(\, \sum_\nu  \|BA^{\theta_0}_\nu e_\la\|_{L^{6}(M)}^{6}\Bigr)^{1/6}.
\end{aligned}
\end{equation}
Similarly, by \eqref{2.33}, the third term can be bounded by 
\begin{equation}\label{2.44e}
\begin{aligned}
 \la^{-\frac16+\frac89\delta_0}\Bigl(\, \sum_\mu \sum_{|\nu-\mu|\le C\theta_1}\|A^{\theta_0}_\nu e_\la\|_{L^2}^2\|e_\la\|_{L^2}^{4}\Bigr)^{1/6} \lesssim  \la^{-\frac16+\frac89\delta_0}.
\end{aligned}
\end{equation}
Using \eqref{n3} for $A^{\theta_1}_\mu$, and \eqref{2.33} for $\tilde A^{\theta_1}_\mu$, the second term on the right side of \eqref{2.44c} is bounded by 
\begin{equation}\label{2.44f}
\begin{aligned}
 \la^{\frac{\delta_0}{6}+\e}& \Bigl(\, \sum_\mu \|A^{\theta_1}_\mu e_\la\|^{2}_{L^\infty(M)}\|\tilde A^{\theta_1}_\mu e_\la\|^{4}_{L^2(M)}\Bigr)^{1/6}\\
 &\lesssim  \la^{\frac{\delta_0}{6}+\frac{1}{3}(\frac{1}{2}-\delta_\infty)(1-\delta_1)+\e} \Bigl(\, \sum_\mu \|\tilde A^{\theta_1}_\mu e_\la\|^{2}_{L^2(M)}\Bigr)^{1/6} \big(\, \sup_\mu \|\tilde A^{\theta_1}_\mu e_\la\|^{\frac13}_{L^2(M)}\big)\\
 &\lesssim  \la^{\frac{\delta_0}{6}-\frac{\eta\delta_1}{3}+\frac{1}{3}(\frac{1}{2}-\delta_\infty)(1-\delta_1)+\e}.
\end{aligned}
\end{equation}
Note that, by the definition of $\tilde A^{\theta_1}_\mu$ in \eqref{aw}, it can be written as a finite sum of operators of the form $A^{\theta_1}_\mu$. Therefore, the last inequality above follows from \eqref{i3a}.

Still because of our assumption on $\delta_0$, the right hand side of \eqref{2.44e} is dominated by the right hand side of \eqref{2.44f}.
Thus, if we combine \eqref{2.44c}-\eqref{2.44f}, we have 
\begin{equation*}
\begin{aligned}
    \Bigl(\, \sum_\mu \|BA^{\theta_1}_\mu e_\la\|_{L^{6}(M)}^{6}\Bigr)^{1/6}&\lesssim  \Bigl(\, \sum_\nu  \|B A^{\theta_0}_\nu e_\la\|_{L^{6}(M)}^{6}\Bigr)^{1/6}+  \la^{\frac{\delta_0}{6}-\frac{\eta\delta_1}{3}+\frac{1}{3}(\frac{1}{2}-\delta_\infty)(1-\delta_1)+\e}.
\end{aligned}
\end{equation*}
Finally, using the induction hypothesis, we conclude that
\begin{equation*}
       \|Be_\la\|_{L^6(M)}\lesssim \Bigl(\, \sum_\nu  \|B A^{\theta_0}_\nu e_\la\|_{L^{6}(M)}^{6}\Bigr)^{1/6}+  \la^{\frac{\delta_0}{6}-\frac{\eta\delta_1}{3}+\frac{1}{3}(\frac{1}{2}-\delta_\infty)(1-\delta_1)+\e}+\la^{\frac{1}{3}(\frac{1}{2}-\delta_\infty)+\e_0}
\end{equation*}
so that \eqref{n2'} holds as long as $\frac{\delta_0}{6}-\frac{\eta\delta_1}{3}+\frac{1}{3}(\frac{1}{2}-\delta_\infty)(1-\delta_1)+\e\leq\frac{1}{3}(\frac{1}{2}-\delta_\infty)+\e_0$.
If we take
$$\delta_1=\e_0\sum_{i=0}^n \left(1+2(\eta-\delta_\infty)\right)^i,$$ 
then \eqref{n2'} holds with 
\[
\delta_0\leq \min\left\{\e_0\sum_{i=1}^{n+1}\left(1+2(\eta-\delta_\infty)\right)^i + 6(\e_0-\e), \;\tfrac{3}{8}(1-{\delta_\infty})\right\}.
\]
By choosing $0<\e<\e_0$ small enough such that $\e_0<6(\e_0-\e)$, we complete the induction step.
\end{proof}

\begin{proof}[Proof of Theorem \ref{thma1}]
By using the remark below \eqref{2.12} and applying Theorem~\ref{thmb}, it suffices to estimate 
\begin{equation*}
 \Bigl(\, \sum_\nu  \|B A^{\theta_0}_\nu e_\la\|_{L^{6}(M)}^{6}\Bigr)^{1/6}= \Bigl(\, \sum_\nu  \|BA^{\theta_0}_\nu \sigma_\la e_\la\|_{L^{6}(M)}^{6}\Bigr)^{1/6}
\end{equation*}
with $\theta_0=\la^{-\delta_0}$. If we use \eqref{commute} and repeat the arguments in \eqref{260}-\eqref{3}, it is not hard to show 
\begin{equation*}
 \Bigl(\, \sum_\nu  \|B A^{\theta_0}_\nu \sigma_\la e_\la\|_{L^{6}(M)}^{6}\Bigr)^{1/6} \lesssim  \Bigl(\, \sum_\nu  \|B\sigma_\la A^{\theta_0}_\nu  e_\la\|_{L^{6}(M)}^{6}\Bigr)^{1/6}+\la^{(-\frac16+\frac89\delta_0)}.
\end{equation*}
By \eqref{2.9}, the $L^q$ spectral cluster bounds of Sogge \cite{sogge881} along with \eqref{2.33}, we have 
\begin{equation*}
\Bigl(\, \sum_\nu  \|B\sigma_\la A^{\theta_0}_\nu  e_\la\|_{L^{6}(M)}^{6}\Bigr)^{1/6}\lesssim 
\la^{\frac16}\Bigl(\, \sum_\nu  \|A^{\theta_0}_\nu  e_\la\|_{L^{2}(M)}^{6}\Bigr)^{1/6}\lesssim \la^{\frac16}  \left(\sup_\nu\|A_\nu^{\theta_0}  e_\la\|_{L^2(X)}\right)^{\frac23}.
\end{equation*}
This completes the proof of Theorem \ref{thma1}.
\end{proof}

\section{Improved microlocal Kakeya-Nikodym estimates}\label{sec 3}
In this section, we prove Theorem \ref{thm: improved KN norm} for a Hecke–Maass form $\psi$ with spectral parameter $\lambda\gg1$, on a compact arithmetic congruence hyperbolic surface $X$. Assume that $\la_1=\sqrt{\frac14+\la^2}$ so that $\la_1\sim \la$. 
Then $\sigma_{\la_1}\psi=\psi$, and it suffices to estimate $\| A^{\theta_0}_\nu \sigma_{\la_1}  \psi\|_{L^{2}(X)}$ and $\| A^{\theta_0}_\nu \sigma_{\la_1}  \psi\|_{L^{\infty}(X)}$.

\subsection{Hecke-Maass forms}\label{sec: arith hyper sur}
We first review the definitions for $X$ and $\psi$.
\subsubsection{Quaternion algebras}
Let $D=\left(\frac{a,b}{\BQ}\right)$ be a quaternion division algebra over $\BQ$, where $a,b\in\BZ$ are squarefree and $a>0$. We choose a basis $1,\omega,\Omega,\omega\Omega$ for $D$ over $\BQ$ that satisfies $\omega^2=a$, $\Omega^2=b$ and $\omega\Omega+\Omega\omega=0$. We denote the reduced norm and trace on $D$ by $\mathrm{nrd}(\alpha)=\alpha\overline{\alpha}$ and $\mathrm{trd}(\alpha)=\alpha+\overline{\alpha}$, where $\alpha\mapsto \overline{\alpha}$ is the standard involution on $D$. We let $R$ be a maximal order in $D$. For any integer $m\geq1$, we let
\[
 R(m)=\{\alpha\in R\mid \mathrm{nrd}(\alpha)=m \}.
\]
The group $R(1)$ of elements of reduced norm 1 in $R$ acts on $R(m)$ by multiplication on the left, and $R(1)\backslash R(m)$ is finite (see \cite{Eic55}). We fix an embedding $\iota: D\to M_2(\BQ(\sqrt{a}))$, into the 2-by-2 matrices with entries in $\BQ(\sqrt{a})$, defined by
\begin{align}\label{eq: embedding of quaternion}
\iota(\alpha)=\begin{pmatrix}
   \xi&b\eta\\\overline{\eta}&\overline\xi
\end{pmatrix},
\end{align}
where \(\alpha=x_0+x_1\omega+(x_2+x_3\omega)\Omega=\xi+\eta\Omega.\) Let $\BH = \{z\in\BC\mid \Im(z)>0 \}$ be the upper half plane. The Lie group $\GL^+(2,\BR)$, consisting of 2-by-2 real matrices with positive determinant, acts on $\BH$ via the fractional linear transformation:
\begin{align*}
    g\cdot z=\frac{az+b}{cz+d},\qquad\text{ where }g=\begin{pmatrix}
        a&b\\c&d
    \end{pmatrix}\in \GL^+(2,\BR)\text{ and }z\in\BH.
\end{align*}
We define the lattice $\Gamma=\iota(R(1))\subset \SL(2,\BR)$, which is cocompact as we assumed $D$ to be a division algebra. We let $X=\Gamma\backslash\BH$ be the corresponding compact arithmetic hyperbolic surface, equipped with the standard volume form $y^{-2}dxdy$ for $x+iy\in\BH$. 

\subsubsection{Hecke operators}\label{sec: Hecke}
We define the Hecke operators $T_n:L^2(X)\to L^2(X)$, $n\geq1$, by
\[
T_nf(z)=\frac{1}{\sqrt{n}}\sum_{\alpha\in R(1)\backslash R(n)} f(\iota(\alpha)z).
\]
There is a positive integer $q$, depending on $R$, such that for any positive integers $m$ and $n$ so that $(m,q)=(n,q)=1$, $T_n$ has the following properties:
\begin{align}
    &T_n=T_n^*,\qquad\text{ that is, }T_n\text{ is self-adjoint},\notag\\
    &T_mT_n=T_nT_m=\sum_{d|(m,n)} T_{mn/d^2}.\label{eq: Hecke relation}
\end{align}

\subsubsection{Lie groups}
We let $G=\mathrm{PSL}(2,\BR)$ and let $K=\mathrm{PSO}(2)$, which is a maximal compact subgroup of $G$. Let $N=\left\{ \left(\begin{smallmatrix}
    1&*\\0&1
\end{smallmatrix}\right) \right\}$ be the unipotent subgroup of $G$, and $A$  the diagonal subgroup, with parameterizations
\[
n(x_1)=\begin{pmatrix}
    1&x_1\\0&1
\end{pmatrix},\qquad a(x_2)=\begin{pmatrix}
    e^{x_2}&0\\0&1
\end{pmatrix},\qquad x_1,x_2\in\BR.
\]
We denote the Lie algebra of $G$ by $\fg$ and equip $\fg$ with the norm
\begin{align*}
    \|\cdot\|:\begin{pmatrix}
        X_1&X_2\\X_3&-X_1
    \end{pmatrix}\mapsto \sqrt{X_1^2+X_2^2+X_3^2}.
\end{align*}
This norm defines a left-invariant metric on $G$, which we denote by $d$. 

\subsubsection{Hecke-Maass forms}

We let $\Delta_g$ be the Laplace-Beltrami operator on $\BH$ and $X$, which is induced from the standard hyperbolic Riemannian metric. Let $\psi\in L^2(X)$ be a Hecke-Maass form that is an eigenfunction of $\Delta_g$ and the operators $T_n$ with $(n,q)=1$. We let $\lambda(n)$ be the Hecke eigenvalues of $\psi$ and $\lambda$ be its spectral parameter, so that
\begin{equation}\label{def}
    \begin{aligned}
           &T_n\psi=\lambda(n)\psi,\\
    &\Delta_g\psi+(\frac{1}{4}+\lambda^2)\psi=0. 
    \end{aligned}
\end{equation}
We assume that $\| \psi \|_{L^2(X)}=1$ with respect to the hyperbolic volume on $X$ and assume that $\lambda\gg1$. Note that because $\Delta_g$ and $T_n$ with $(n,q)=1$ are self-adjoint, we may assume that $\psi$ is real-valued. For functions on $X$, we will also think of them as functions on $G$ that are left $\Gamma$-invariant and right $K$-invariant.

\subsection{Reducing the pseudo-differential operator}\label{sec: reducing}
We may first identify $X=\Gamma\backslash\mathbb{H}$ with a fundamental domain $\cD\subset \BH^2=\{n(x_1)a(x_2)\cdot i\mid(x_1,x_2)\in\BR^2\}$.
Recall that in Section \ref{sec: microlocal decomposition} the pseudo-differential operator $A^{\theta_0}_\nu$ is defined using a local coordinate chart $\Omega$. We define it explicitly on $X$ using the Iwasawa decomposition $G=NAK$ as follows. 
Let $r>0$ be a fixed constant, which may be taken to be small enough, and
\[
\tilde\Omega=\{x=(x_1,x_2):\, |x_1|,|x_2|<r\}\subset\BR^2.
\]
We may assume there exists an isometry $g_0\in G$ such that
\[
\Omega=\{g_0n(x_1)a(x_2)\cdot i:\,(x_1,x_2)\in\tilde\Omega \}
\]
is contained in the fundamental domain $\cD$, and the $x$-variable of $A^{\theta_0}_\nu$ is supported in
$$\{x=(x_1,x_2):\, |x_1|,|x_2|<0.1r\}\subset\tilde\Omega.$$
We denote the diffeomorphism from $\tilde\Omega$ onto $\Omega$ by $\kappa$ and 
\begin{align}\label{eq: diffeo kappa}
    \kappa(x)=\kappa((x_1,x_2))=g_0n(x_1)a(x_2)\cdot i.
\end{align}
By abuse of the notation, we still write the corresponding diffeomorphism on the cotangent bundles as $\kappa$. If we let $(x_0,\xi_0)\in S^*\tilde\Omega$, with $x_0=(0,0)$ and $\xi_0=(0,1)$, it may be seen that this local coordinate chart does satisfy the assumption \eqref{eq: assumption on local chart} under the hyperbolic metric (see \eqref{eq: hyperbolic metric under Iwasawa} below), so that the construction in Section \ref{sec: microlocal decomposition} can be applied here.
By choosing $g_0$ appropriately with $d(g_0,e)\lesssim1$, we may assume the center geodesic $\gamma_\nu \in S^*\Omega$ of $A^{\theta_0}_\nu$ (see the definition of $\gamma_\nu$ under \eqref{b.46}) is
$$ \kappa\big(\{(x,\xi)=((0,x_2), (0,1))\in S^*\tilde\Omega:\, |x_2|< 0.1r\}\big).$$
In the rest of the section, we treat the variables $(x,\xi)$ of $A^{\theta_0}_\nu$ as in $T^*\tilde\Omega=\tilde\Omega\times\BR^2$. By \eqref{b.48} and the support property of the symbol $A^{\theta_0}_\nu$ in \eqref{qnusymbol}, we have
\begin{equation}\label{3.1a}
    \begin{split}
        \supp A^{\theta_0}_\nu(x,\xi)\subset  \Big\{(x,\xi)\in T^*\tilde\Omega:\;&|x_1|\lesssim \la^{-\delta_0}, \,|x_2|<0.1r, \\ & p(x,\xi)\in (\frac{\la}{4}, 4\la),\, |\frac{\xi}{p(x,\xi)}-(0,1)|\lesssim \la^{-\delta_0}\Big\}.
    \end{split}
\end{equation}
Note that in the above coordinates the hyperbolic metric on $(x_1,x_2)\in \Omega_0$ is given by 
\begin{align}\label{eq: hyperbolic metric under Iwasawa}
    ds^2=e^{-2x_2}dx_1^2+dx_2^2.
\end{align}
Thus the principal symbol of $\sqrt{-\Delta_g}$ is $p(x,\xi)=\sqrt{e^{2x_2}\xi_1^2+\xi_2^2}$. Hence, by \eqref{3.1a}, in the support of $A^{\theta_0}_\nu(x,\xi)$, 
the component $\xi_2$ corresponds to the radial direction, while $\xi_1$ corresponds to the angular direction. More precisely, for $\la\gg1$,
\begin{equation}\label{3.1}
\supp A^{\theta_0}_\nu(x,\xi) \subset \Big\{(x,\xi)\in T^*\tilde\Omega:\,|x_1|\lesssim \la^{-\delta_0}, |x_2|< 0.1r, \, \xi_2\in (\frac{\la}{5}, 5\la), |\xi_1|\lesssim \la^{1-\delta_0}\Big\}.
\end{equation}
Moreover, by recalling \eqref{b.45a}, we also have
\begin{equation}\label{3.2}
\bigl|\partial_x^\sigma \partial_\xi^\gamma A^\theta_\nu(x,\xi)\bigr| \lesssim \la^{\delta_0|\sigma|-(1-\delta_0)|\gamma|}.
\end{equation}

Note that if $\la_1=\sqrt{\frac14+\la^2}$ with $\la\gg 1$, then $|\la_1-\la|\lesssim \la^{-1}$. We fix a nonnegative function $\chi\in C_0^\infty(\BR)$, so that $\chi=1$ on $(-0.2r,0,2r)$ and $\chi$ vanishes outside $(-0.3r,0.3r)$. Thus, we can use \eqref{2.2a} to see that the kernel $A^{\theta_0}_\nu \sigma_{\la_1} (\kappa(x),\kappa(y))$ of the operator $A^{\theta_0}_\nu \sigma_{\la_1} $ is, modulo $O_N(\la^{-N})$ terms,
\begin{align}\label{3.9}
\begin{split}
    \frac{\lambda^{1/2}}{(2\pi)^2}\int_{T^*\tilde\Omega} e^{i\langle x-z, \xi \rangle }A^{\theta_0}_\nu(x,\xi) e^{i\lambda d_g(\kappa(z),\kappa(y))} \chi(z_1)\chi(z_2)a(\kappa(z),\kappa(y),\lambda) dz d\xi.
    \end{split}
\end{align}
Here $x,y\in\tilde\Omega$, $d_g$ represents the distance function on $\BH$, and $a(\kappa(z),\kappa(y),\la)$ is a smooth function with the same properties as the one in \eqref{2.2a}. 
\begin{rmk}
    Here we insert the extra cutoff $\chi(z_1)\chi(z_2)$ due to the rapid decay off-diagonal estimate \eqref{ke} for the Schwartz kernel for $A_\nu^{\theta_0}$. Moreover, we have restricted the domain of the second variable $\kappa(y)$ in $\Omega$. This is due to the support condition \eqref{2.2a} on $a(\cdot,\cdot,\la)$ (we may choose $\delta$ in \eqref{2.2} to be sufficiently small).
\end{rmk}
Define $\sA$ to be the pseudo-differential operator with the compound symbol
\begin{align}\label{eq: symbol A}
    \begin{split}
        &\sA(x,y,\xi)=\chi\left(\frac{\la^{\delta_0}x_1}{C_1}\right)\chi\left(\frac{x_2}{2}\right)\chi\left(\frac{\la^{\delta_0}y_1}{C_1}\right)\chi(y_2)\chi\left(\frac{\la^{\delta_0-1}\xi_1}{C_1}\right)\chi\left(\frac{\la^{2\delta_0-1}(\xi_2-\la)}{C_1}\right),
    \end{split}
\end{align}
with $C_1$ a sufficiently large positive constant.
Given $\phi\in C^\infty(X)$, $\sA\phi$ is supported in $\Omega$ and is given by the integral
\begin{align}\label{eq: defn of operator A}
    \sA \phi (\kappa(x))=(2\pi)^{-2} \int_{T^*\tilde\Omega} e^{i\langle x-y, \xi \rangle }\sA(x,y,\xi) \phi(\kappa(y))  dyd\xi.
\end{align}
\begin{rmk}
    The operator $\sA$ is the analogue of a wave packet in Euclidean harmonic analysis. 
\end{rmk}
We will need the following lemma relating $A^{\theta_0}_\nu$ and $\sA$.
\begin{lemma}\label{lem: reduction of pdo}
    If $C_1\gg 1$ is chosen sufficiently large, we have, as operators between $L^p(X)\to L^q(X)$, for any $1\le p,q\le \infty$, 
    \begin{equation}\label{ao}
        A^{\theta_0}_\nu \sigma_{\la_1}= A^{\theta_0}_\nu  \sA \sigma_{\la_1}+O_N(\la^{-N}), \quad \forall \, N>0.
    \end{equation}
\end{lemma}

\begin{proof}
By \eqref{3.9} and the Fourier and inverse Fourier transforms, the kernel $A^{\theta_0}_\nu \sigma_{\la_1} (\kappa(x),\kappa(y))$ is, modulo $O(\la^{-N})$ terms, 
\begin{equation*}
\begin{aligned}
     \frac{\lambda^{1/2}}{(2\pi)^4}\iint_{T^*\tilde\Omega} e^{i\langle x-w,\xi\rangle}A^{\theta_0}_\nu(x,\xi)e^{i\langle w-z, \eta \rangle } e^{i\lambda d_g(\kappa(z),\kappa(y))} &\chi(z_1)\chi(z_2)a(\kappa(z),\kappa(y),\lambda) dw d\xi dzd\eta  \\
     =\frac{\lambda^{1/2}}{(2\pi)^4}\iint_{T^*\tilde\Omega} e^{i\langle x-w,\xi\rangle}A^{\theta_0}_\nu(x,\xi)e^{i\langle w-z, \eta \rangle } e^{i\lambda d_g(\kappa(z),\kappa(y))} &\chi(w_1/2)\chi(w_2/2)\chi(z_1)\chi(z_2)\\
     &a(\kappa(z),\kappa(y),\lambda)  dwd\xi dzd\eta . 
\end{aligned}
\end{equation*}
Let us fix  $\tilde\rho\in C_0^\infty(1/7, 7)$ with $\tilde \rho=1$ on $(1/6, 6)$. Then, by integrating by parts in $w_1,w_2$ and use \eqref{3.1}, we have, modulo $O(\la^{-N})$ terms, $A^{\theta_0}_\nu \sigma_{\la_1} (\kappa(x),\kappa(y))$ is
\begin{multline*}
     \frac{\lambda^{1/2}}{(2\pi)^4}\iint_{T^*\tilde\Omega}  e^{i\langle x-w,\xi\rangle}A^{\theta_0}_\nu(x,\xi)e^{i\langle w-z, \eta \rangle }\chi\left(\frac{\la^{\delta_0-1}\eta_1}{C_1}\right)\tilde\rho(\eta_2/\la)\\
     \cdot e^{i\lambda d_g(\kappa(z),\kappa(y))} \chi(w_1/2)\chi(w_2/2)\chi(z_1)\chi(z_2)a(\kappa(z),\kappa(y),\lambda)  dw d\xi dzd\eta, 
\end{multline*}
for $C_1$ sufficiently large.
Next, by integrating by parts in $\xi_1$ and then in $\eta_1$ (we may integrate by parts in $\xi_1$ again to make the widths of the cutoffs for $z_1,w_1$ be the same), and using \eqref{3.2} and $0<\delta_0<\frac{1}{2}$, we have, modulo $O(\la^{-N})$ terms, 
\begin{equation}\label{3.18}
\begin{aligned}
         &\qquad A^{\theta_0}_\nu \sigma_{\la_1} (\kappa(x),\kappa(y)) \\
     &= \frac{\lambda^{1/2}}{(2\pi)^4}\iint_{T^*\tilde\Omega} e^{i\langle x-w,\xi\rangle}A^{\theta_0}_\nu(x,\xi)e^{i\langle w-z, \eta \rangle }\chi\left(\frac{\la^{\delta_0-1}\eta_1}{C_1}\right)\tilde\rho(\eta_2/\la)e^{i\lambda d_g(\kappa(z),\kappa(y))}\\
     &\qquad\qquad\cdot\chi\left(\frac{\la^{\delta_0}w_1}{C_1}\right) \chi(w_2/2) \chi\left(\frac{\la^{\delta_0}z_1}{C_1}\right)   \chi(z_2)a(\kappa(z),\kappa(y),\lambda)   dwd\xi dzd\eta,
\end{aligned}
\end{equation}
for $C_1$ sufficiently large.

To introduce the desired cutoff in $\eta_2$, we will need the following lemma.

\begin{lemma}\label{lem: der of dg}
    Given $z,y\in\tilde\Omega$ so that $a(\kappa(z),\kappa(y),\lambda)\neq0$, then
    \[
    |\partial_{z_1}d_g(\kappa(z),\kappa(y))|\sim |z_1-y_1|,
    \]
    and for any integer $n\ge2$
    \[
    |\partial^n_{z_1}d_g(\kappa(z),\kappa(y))|\lesssim_n 1 .
    \]
    If, furthermore, $z,y$ satisfy $|z_1|,|y_1|\lesssim\la^{-\delta_0}$, then
    $$\big|\partial_{z_2}d_g(z,y)\big |=1+O(\la^{-2\delta_0}),$$
    and for any integer $n\ge2$ we have
    \[
    \big|\partial^n_{z_2}d_g(z,y)\big |\lesssim_n\la^{-2\delta_0}.
    \]
\end{lemma}

\begin{proof}
    Recall the distance function on the hyperbolic plane $\BH$ is given by 
    \begin{equation}\label{ds}
        \cosh \Big(d_g(\kappa(z),\kappa(y))\Big)=\cosh(z_2-y_2)+\frac{(z_1-y_1)^2}{2e^{z_2+y_2}}.
    \end{equation}
    Taking $\partial_{z_1}$, we obtain
    \begin{align*}
        \sinh\Big(d_g(\kappa(z),\kappa(y))\Big)\partial_{z_1}d_g(\kappa(z),\kappa(y))=\frac{z_1-y_1}{e^{z_2}e^{y_2}}.
    \end{align*}
    The first derivative result of $z_1$ follows by using $d_g(\kappa(z),\kappa(y))\sim1$ when $a(\kappa(z),\kappa(y),\la)\neq 0$. The higher derivatives are bounded trivially.

    Now we apply $\partial_{z_2}$ to \eqref{ds} to get, with $|z_1|,|y_1|\lesssim\la^{-\delta_0}$,
    \begin{align*}
        \sinh\Big(d_g(\kappa(z),\kappa(y))\Big)\partial_{z_2}d_g(\kappa(z),\kappa(y))=\sinh(z_2-y_2)-\frac{(z_1-y_1)^2}{2e^{z_2+y_2}}.
    \end{align*}
    Using $|z_1|,|y_1|\lesssim\la^{-\delta_0}$ and $d_g(\kappa(z),\kappa(y))\sim1$, it follows from \eqref{ds} that $|z_2-y_2|\sim1$ and
    \[
    \sinh(d_g(\kappa(z),\kappa(y)))=\sinh(|z_2-y_2|)(1+O(\la^{-2\delta_0})),
    \]
    which proves the result for the first derivative. To bound the higher derivatives, we note that
    \begin{align*}
        \partial_{z_2}d_g(\kappa(z),\kappa(y))=\frac{\sinh(z_2-y_2)-\frac{(z_1-y_1)^2}{2e^{z_2+y_2}}}{\sqrt{\left(\cosh(z_2-y_2)+\frac{(z_1-y_1)^2}{2e^{z_2+y_2}}\right)^2-1}},
    \end{align*}
    and for any $w\in\BC$ with $|w|=O(1)$ the real part inside the square root is
    \begin{multline*}
        \Re\left(\left(\cosh(w-y_2)+\frac{(z_1-y_1)^2}{2e^{w+y_2}}\right)^2-1\right)=\Re\left(\sinh^2(w-y_2)+O(\lambda^{-2\delta_0})\right)\\
        =\frac{\cosh(2\Re(w-y_2))\cos(2\Im (w))-1}{2}+O(\lambda^{-2\delta_0}) = \sinh^2(\Re(w-y_2)) +O(\Im(w)^2)+O(\lambda^{-2\delta_0}).
    \end{multline*}
    Recall that we have $|z_2-y_2|\sim1$. Hence, for $C>0$ and for any $w\in\BC$ with $|w-z_2|\leq C$, we have
    \[
    \Re\left(\left(\cosh(w-y_2)+\frac{(z_1-y_1)^2}{2e^{w+y_2}}\right)^2-1\right)=\sinh^2(\Re(w-z_2)+(z_2-y_2))+O(C^2)+O(\lambda^{-2\delta_0}) \gg1,
    \]
    which holds when $C>0$ is chosen to be sufficiently small. Thus, 
    \[\frac{\sinh(w-y_2)-\frac{(z_1-y_1)^2}{2e^{w+y_2}}}{\sqrt{\left(\cosh(w-y_2)+\frac{(z_1-y_1)^2}{2e^{w+y_2}}\right)^2-1}}\]
    can be extended to be holomorphic, in $w$, on the ball of radius $C$ centered at $z_2$. By applying Cauchy's integral formula, we obtain, for any $n\geq1$,
    \begin{align*}
        \partial^{n+1}_{z_2}d_g(\kappa(z),\kappa(y))=\partial^n_{z_2} \left(\partial_{z_2}d_g(\kappa(z),\kappa(y))-\sgn (z_2-y_2))\right) \ll_{C,n} \lambda^{-2\delta_0},
    \end{align*}
    which completes the proof.
\end{proof}

Now we continue the proof of Lemma~\ref{lem: reduction of pdo}. 
If we further integrate \eqref{3.18} by parts in $z_1$, using $|\eta_1|\lesssim\la^{1-\delta_0}$ from \eqref{3.1}, $0<\delta_0<\frac{1}{2}$, and Lemma \ref{lem: der of dg} (and the trivial bounds for the higher derivatives), we have, modulo $O(\la^{-N})$ terms, $A^{\theta_0}_\nu \sigma_{\la_1} (\kappa(x),\kappa(y))$ is
\begin{equation}\label{3.19}
\begin{aligned}
       &\frac{\lambda^{1/2}}{(2\pi)^4}\iint_{T^*\tilde\Omega} e^{i\langle x-w,\xi\rangle}A^{\theta_0}_\nu(x,\xi)e^{i\langle w-z, \eta \rangle }\chi\left(\frac{\la^{\delta_0-1}\eta_1}{C_1}\right)\tilde\rho(\eta_2/\la) e^{i\lambda d_g(\kappa(z),\kappa(y))} \\
     &\qquad\quad\cdot \chi\left(\frac{\la^{\delta_0}w_1}{C_1}\right) \chi(w_2/2)\chi\left(\frac{\la^{\delta_0}z_1}{C_1}\right)\chi\left(\frac{\la^{\delta_0}y_1}{C_1}\right)   \chi(z_2)a(\kappa(z),\kappa(y),\lambda) dw d\xi dzd\eta, 
\end{aligned}
\end{equation}
for $C_1$ sufficiently large. 
Now integrating by parts in $z_2$ from \eqref{3.19}, and using Lemma \ref{lem: der of dg} and $0<\delta_0<\frac{1}{2}$, we have, modulo $O(\la^{-N})$ terms, $A^{\theta_0}_\nu \sigma_{\la_1} (\kappa(x),\kappa(y))$ is
\begin{equation}\label{3.20}
\begin{aligned}
     &\frac{\lambda^{1/2}}{(2\pi)^4}\iint_{T^*\tilde\Omega} e^{i\langle x-w,\xi\rangle}A^{\theta_0}_\nu(x,\xi)e^{i\langle w-z, \eta \rangle }\chi\left(\frac{\la^{\delta_0-1}\eta_1}{C_1}\right)\chi\left(\frac{\la^{2\delta_0-1}(\eta_2-\la)}{C_1}\right)e^{i\lambda d_g(\kappa(z),\kappa(y))} \\
     &\cdot \chi\left(\frac{\la^{\delta_0}w_1}{C_1}\right) \chi(w_2/2)\chi\left(\frac{\la^{\delta_0}z_1}{C_1}\right)   \chi(z_2)\chi\left(\frac{\la^{\delta_0}y_1}{C_1}\right)  a(\kappa(z),\kappa(y),\lambda) dw d\xi dzd\eta, 
\end{aligned}
\end{equation}
for $C_1$ sufficiently large. Here we have used that $$\tilde\rho(\eta_2/\la) \chi\left(\frac{\la^{2\delta_0-1}(\eta_2-\la)}{C_1}\right)= \chi\left(\frac{\la^{2\delta_0-1}(\eta_2-\la)}{C_1}\right)$$
when $\la\gg 1 $ is sufficiently large, which may depend on $C_1$. Note that \eqref{3.20} is equal to $$A^{\theta_0}_\nu  \sA \sigma_{\la_1}(\kappa(x),\kappa(y))\chi\left(\frac{\la^{\delta_0}y_1}{C_1}\right).$$ By another integration by part argument in $z_1$, the $\chi\left(\frac{\la^{\delta_0}y_1}{C_1}\right)$ factor can be removed at the expense of an $O(\la^{-N})$ error. This completes the proof of Lemma~\ref{lem: reduction of pdo}. 
\end{proof}

From Lemma \ref{lem: reduction of pdo}, we reduce the pseudo-differential operator $A_\nu^{\theta_0}$ to a simpler operator $\sA$.
\begin{lemma}\label{lem 3.10}
    We have $\| A_\nu^{\theta_0}\psi\|_{L^q(X)}\lesssim\| \sA\psi\|_{L^q(X)} + O_N(\la^{-N})$ for any $2\leq q\leq\infty$ and for any $N>0$. 
\end{lemma}
\begin{proof}
    By \eqref{ao} and the fact $\sigma_{\la_1}\psi=\psi$ we have
    \begin{multline*}
    \|A^{\theta_0}_\nu \psi\|_{L^q}=\|A^{\theta_0}_\nu \sigma_{\la_1}\psi\|_{L^q}    \leq \|A^{\theta_0}_\nu  \sA \sigma_{\la_1}\psi\|_{L^q}+O(\la^{-N}) \leq \|A^{\theta_0}_\nu\|_{L^q\to L^q}  \|\sA \psi\|_{L^q}+O(\la^{-N}),
    \end{multline*}
    which completes the proof by applying the bound $\|A^{\theta_0}_\nu\|_{L^q\to L^q}\lesssim 1$ due to \eqref{2.33}.
\end{proof}

We shall use the following useful facts about $\sA$. We write the $L^2$ pairing on $X$ as $\langle\cdot,\cdot\rangle_X$. We define another compound symbol $\sA^*(x,y,\xi)=e^{x_2-y_2} \sA(y,x,\xi)$, with the corresponding pseudo-differential operator $\sA^*$ defined in a similar way as \eqref{eq: defn of operator A}.

\begin{lemma}\label{adjoint of A}
    $\sA^*$ is the adjoint of $\sA$, i.e. $\langle\sA f,g\rangle_X=\langle f,\sA^* g\rangle_X$ for any $f,g\in L^2(X)$.
\end{lemma}
\begin{proof}
    This follows from the facts that $\sA f$ is supported in $\Omega$, and the volume form for $\kappa(x)\in\Omega$ is $e^{-x_2}dx_1 dx_2$ under our coordinates.
\end{proof}
 
\begin{lemma}\label{ab}
We have $\|\sA\|_{L^p\to L^p}\lesssim 1$ and $\|\sA^*\|_{L^p\to L^p}\lesssim 1$ for any $1\le p\le\infty$. 
\end{lemma}
\begin{proof}
Let $\sA(x,y)$ denote the Schwartz kernel of  $\sA$, with $x,y\in\tilde\Omega$. Then, 
by \eqref{eq: symbol A} and repeated integration by parts, we have
\begin{align}\label{eq: Schwartz kernel for sA}
    \sA(x,y)=
O\bigl(\la^{2-3\delta_0} (1+\la^{1-\delta_0}|x_1-y_1|)^{-N}(1+\la^{1-2\delta_0}|x_2-y_2|)^{-N})\bigr),
\end{align}
and the same holds for $\sA^*$.
Consequently, Lemma~\ref{ab} follows from Young’s inequality.
\end{proof}

\subsection{Pretrace formula}
We fix a real-valued function $h\in C^\infty(\BR)$ of Paley-Wiener type that is nonnegative and satisfies $h(0)=1$. Define $h_\lambda^0$ by $h_\lambda^0(s)=h(s-\lambda)+h(-s-\lambda)$, and let $k_\lambda^0$ be the $K$-bi-invariant function on $\BH$ with Harish-Chandra transform $h_\lambda^0$ (see e.g. \cite[Chap.\,4]{Hel84}). The Paley–Wiener theorem of Gangolli \cite{Gan71} implies that $k_\lambda^0$ is of compact support that may be chosen arbitrarily small. Define $k_\lambda=k_\lambda^0*k_\lambda^0$, which has Harish-Chandra transform $h_\lambda=(h_\lambda^0)^2$.
Let $\{\psi_j\}$ be an orthonormal basis for $L^2(X)$ consisting of Hecke–Maass
forms with $\psi\in\{\psi_j\}$. Let $\la_j$ be the spectral parameter of $\psi_j$. Then we have the pretrace formula by Selberg \cite{Sel56}
\[
K_\la(x,y):=\sum_{\gamma\in\Gamma} k_\la(x^{-1}\gamma y)=\sum_{j} h_\la(\la_j)\psi_j(x)\overline{\psi_j(y)},\qquad \text{ for }x,y\in X.
\]
Let $N\geq1$ be an integer, and let $\{\alpha_n\}_{n=1}^N$ be complex numbers so that $\alpha_n=0$ if $(n,q)\neq1$, both to be chosen later. Define the Hecke operator $\cT = \sum_{n=1}^N \alpha_n T_n$. Recall that after applying $\cT$ to the kernal $K_\la(x,y)$ we have the following amplified pretrace formula in \cite{IS95}:
     \begin{align}\label{eq: amplified pretrace}
         \sum_j h_\la(\lambda_j) \cT\psi_j(x) \overline{\cT\psi_j(y)} =\sum_{m,n\leq N} \alpha_m  \overline\alpha_n\sum_{d|(m,n)} \frac{d}{\sqrt{mn}}\sum_{\gamma\in R(mn/d^2)} k_\la(x^{-1}\gamma y).
     \end{align}

\subsection{Proof of (\ref{i3})}
Given $g\in G$ and $\phi\in C^\infty(X)$, using the notation from Section \ref{sec: reducing}, we define
\[
    I_\sA(\la,\phi,g)=\iint_{\tilde\Omega} \overline{\sA^*\phi}(\kappa(x)) \sA^*\phi(\kappa(y)) k_\la(\kappa(x)^{-1}g\kappa(y))e^{-x_2-y_2} dxdy.
\]
Here $x=(x_1,x_2)$, $dx=dx_1dx_2$, and the factor $e^{-x_2-y_2}$ is from the hyperbolic volume form.

\begin{proposition}\label{prop: amplification inequality}
    For any $\phi\in C^\infty(X)$, we have
    \[
    \left|\langle \sA\cT\psi,\phi\rangle_X\right|^2\leq\sum_{m,n\leq N} |\alpha_m  \alpha_n | \sum_{d|(m,n)} \frac{d}{\sqrt{mn}}\sum_{\gamma\in R(mn/d^2)} |I_\sA(\la,\phi,\gamma)|
    \]
\end{proposition}
\begin{proof}
    We first write \eqref{eq: amplified pretrace} with variables $\kappa(x),\kappa(y)\in\Omega$ to obtain
    \begin{align*}
         \sum_j h_\la(\lambda_j) \cT\psi_j(\kappa(x)) \overline{\cT\psi_j(\kappa(y))}=\sum_{m,n\leq N} \alpha_m  \overline\alpha_n\sum_{d|(m,n)} \frac{d}{\sqrt{mn}}\sum_{\gamma\in R(mn/d^2)} k_\la(\kappa(x)^{-1}\gamma \kappa(y)).
     \end{align*}
     If we integrate the above identity against $\overline{\sA^*\phi}\times\sA^*\phi$ with respect to the volume form, we obtain
    \begin{align*}
         \sum_j h_\la(\lambda_j) \left|\langle \cT\psi_j,\sA^*\phi\rangle_X\right|^2=\sum_{m,n\leq N} \alpha_m  \overline\alpha_n\sum_{d|(m,n)} \frac{d}{\sqrt{mn}}\sum_{\gamma\in R(mn/d^2)} I_\sA(\la,\phi,\gamma).
     \end{align*}
     Here we have used the fact that $\supp\sA^*\phi\subset\Omega$.
     Note that $h_\la(\lambda)\geq 1$ and $h_\la(\lambda_j)\geq 0$ for all $j$. We may therefore drop all terms on the left-hand side except $\psi$. Finally applying the adjointness between $\sA$ and $\sA^*$ from Lemma \ref{adjoint of A} completes the proof.
\end{proof}

To apply the amplification inequality, we shall need the following bounds for $I_\sA(\lambda,\phi,\gamma)$.

\begin{proposition}\label{prop: bound for I_A}
    Let $\phi\in C^\infty(X)$ be such that $\|\phi\|_{L^{2}(X)}=1$. 
     \begin{enumerate}
         \item For any $\gamma\in G$, $I_\sA(\lambda,\phi,\gamma)\lesssim1$.
         \item Fix $0<\epsilon_0\ll\delta_0$. If $d(\gamma,e)\lesssim 1$ and $d(g_0^{-1}\gamma g_0,A)\geq \lambda^{-\delta_0+2\epsilon_0}$, then $I_\sA(\lambda,\phi,\gamma)\lesssim_{\epsilon_0,N} \lambda^{-N}$ for any $N>0$.
     \end{enumerate}
\end{proposition}
\begin{proof}
    For $f\in C^\infty(X)$ we let $$K_\la f(x)=\int_X K_\la(x,y) f(y)dy=\int_\BH k_\la(x^{-1}y)f(y) dy,$$ 
    with $\|K_\la\|_{L^2\to L^2}\lesssim1$ by our construction.
    Therefore, by Cauchy-Schwarz and Lemma \ref{ab},
    \begin{multline*}
        |I_\sA(\lambda,\phi,\gamma)|=\left| \int_X K_\la\sA^*\phi(\gamma^{-1}x)\overline{\sA^*\phi(x)}dx\right|\leq \|K_\la \sA^*\phi\|_{L^2} \|\sA^*\phi\|_{L^{2}}
        \leq \|\sA^*\|_{L^{2}\to L^{2}}^2\|\phi\|_{L^{2}}^2\lesssim1.
    \end{multline*}

    For part \textit{(b)}, we first expand $I_\sA(\lambda,\phi,\gamma)$ by the intgeral representation for $\sA^*$ as
    \begin{multline*}
        I_\sA(\lambda,\phi,\gamma)=\iint_{\tilde\Omega}\iint_{T^*\tilde\Omega} e^{i\langle y-v,\eta\rangle-i\langle x-u, \xi \rangle }\sA^*(x,u,\xi)\sA^*(y,v,\eta)\overline{\phi(u)} \phi(v) \\
        k_\la(\kappa(x)^{-1}\gamma \kappa(y)) e^{-x_2-y_2} dxdy dud\xi dv  d\eta.
    \end{multline*}
    We consider the integral over $x_2,y_2$ in $I_\sA(\lambda,\phi,g)$, that is,
    \begin{align*}
        \iint_{-\infty}^\infty  \chi(x_2)\chi(y_2)e^{i\eta_2 y_2-i\xi_2 x_2}        k_\la(a(-x_2)n(-x_1)g_0^{-1}\gamma g_0n(y_1)a(y_2)) dx_2 dy_2.
    \end{align*}
    Here we use the relation $\sA^*(x,y,\xi)=e^{x_2-y_2} \sA(y,x,\xi)$, and the explicit formulas \eqref{eq: symbol A} for $\sA$ and \eqref{eq: diffeo kappa} for $\kappa$.
    Let $\beta=(0.3r)C_1\la^{1-2\delta_0}\sim \la^{1-2\delta_0}$ and $g=n(-x_1)g_0^{-1}\gamma g_0n(y_1)$.
    From the support of $\sA^*$, we may assume that $ |x_1|,|y_1|\lesssim \la^{-\delta_0}$ and $\xi_2,\eta_2\in [\la-\beta,\la+\beta]$.  Therefore, by applying H\"older's inequality to the integrals over the rest variables, it suffices to show
    \begin{align}\label{eq: int I}
        \iint_{-\infty}^\infty  \chi(x_2)\chi(y_2)e^{i\eta_2 y_2-i\xi_2 x_2}        k_\la(a(-x_2)ga(y_2)) dx_2 dy_2\lesssim_N \la^{-N}.
    \end{align}
    Since $d(g_0^{-1}\gamma g_0,A)\geq \lambda^{-\delta_0+2\epsilon_0}$ and $ x_1,y_1\lesssim \la^{-\delta_0}$, it may be seen that $d({g},A)\geq \lambda^{-1/2+\epsilon_0}\beta^{1/2}$ provided $\la\gg1$. Then \eqref{eq: int I} follows from the following lemma.
\end{proof}

\begin{lemma}\label{lem: Marshall rapid decay}
    Given $0<\epsilon'<1/2$, suppose that $\lambda^{\epsilon'}\leq\beta \leq \lambda^{1-\epsilon'}$, and $\xi_2,\eta_2\in[\lambda-\beta,\lambda+\beta]$. If $\epsilon_0>0$, $d(g,e)\lesssim1$ and $d(g,A)\geq \lambda^{-1/2+\epsilon_0}\beta^{1/2}$, we have
    \begin{align*}
        \iint_{-\infty}^\infty  \chi(x_2)\chi(y_2)e^{i\eta_2 y_2-i\xi_2 x_2}        k_\la(a(-x_2)ga(y_2)) dx_2 dy_2\lesssim_{\epsilon_0,N} \la^{-N},
    \end{align*}
    for any $N>0$.
\end{lemma}
\begin{proof}
    This follows immediately from a result of Marshall \cite[Proposition 6.9]{Mar16} by the inverse Harish-Chandra transform and the rapid decay of $h_\la$ away from $\la$.
    Note that in \cite{Mar16} it is required that $\beta\leq\lambda^{2/3}$. However, this assumption is not necessary, and $\lambda^{2/3}$ could be replaced with $\lambda^{1-\epsilon'}$ (this is also mentioned in \cite[Remark p.466]{Mar16}). 
\end{proof}

We also need an input on the estimation of Hecke returns, which counts how many times the Hecke operators map a geodesic segment close to itself. We define, for $g\in G$ and $0\leq\kappa\leq1$,
\begin{align*}
     \sM(g,n,\kappa)=\{\eta\in R(n)\mid d(g^{-1}\eta g,e)\leq1, d(g^{-1}\eta g,A)\leq\kappa\}.
\end{align*}

\begin{lemma}\label{lem: Hecke return new}
     There exists a constant $C>0$ such that the following holds. If $d(g,e)\leq1$ and $\kappa\le Cn^{-1}$, we have the bound
    \[         |\sM(g,n,\kappa)|\ll_\epsilon n^\epsilon.     \]
\end{lemma}
\begin{proof}
    We first show that elements in \(\sM(g,n,\kappa)\) are commutative with each other; as a consequence, $\sM(g,n,\kappa)$ must be contained in a quadratic field over $\BQ$ in the quaternion division algebra $A$. Suppose otherwise and we have $xy-yx\neq 0$ for some $x, y\in\sM(g,n,\kappa)$. There exist $u,v\in gAg^{-1}$ so that $\det u=\det v =1$, $uv-vu=0$ and
    \[
    \frac{1}{\sqrt{n}}\iota(x)=u+O(\kappa),\qquad \frac{1}{\sqrt{n}}\iota(y)=v+O(\kappa).
    \]
    Therefore, \(0\neq xy-yx= O(n\kappa).\)
    Using $n\kappa\leq C$ and choosing $C$ to be sufficiently small, we get a contradiction with $xy-yx\in R$.

    Now we let $E$ be the quadratic field containing $\sM(g,n,\kappa)$. The reduced norm $\mathrm{nrd}$ restricted to $E$ agrees with the norm map on $E$. Therefore, if $\eta\in\sM(g,n,\kappa)$, then $\eta\in E\cap R$ has norm $n$. Moreover, by the boundedness condition on $g$ and $\eta$, the image of $\eta$ under any archimedean embedding of $E$ must be $\lesssim n^{1/2}$. It may be seen that the number of these algebraic numbers must be $\lesssim_\e n^\e$.
\end{proof}

\subsubsection{Amplifier}\label{subsubsection amplifier}
Let $N\geq1$ be a large positive integer to be chosen later. If $1< p\leq\sqrt{N}$ is a prime so that $(p,q)=1$, \eqref{eq: Hecke relation} implies the following relation for the Hecke eigenvalues:
\begin{align*}
     \lambda(p)^2-\lambda(p^2)=1.
\end{align*}
For such a prime $p$, the above relation implies that if $|\lambda(p)|<1/2$ then $|\lambda(p^2)|>3/4$. We set
\begin{align*}
     \begin{cases}
         \alpha_p= \frac{\lambda(p)}{|\lambda(p)|},\;\alpha_{p^2}=0,\qquad&\text{ if }|\lambda(p)|\geq 1/2,\\
         \alpha_p= 0,\;\alpha_{p^2}=\frac{\lambda(p^2)}{|\lambda(p^2)|},\qquad&\text{ if }|\lambda(p)|< 1/2.
     \end{cases}
\end{align*}
We set $\alpha_n=0$ for all the other $n$. It may be seen that
\begin{align}
     \sum_{n\leq N}|\alpha_n| = \sum_{n\leq N}|\alpha_n|^2=\sum_{\substack{1<p\leq\sqrt{N}\\ (p,q)=1}} 1 \lesssim N^{1/2}.\label{eq: moment for alpha}
\end{align}
Let $\cT = \sum_{n=1}^N \alpha_n T_n$ be our amplifier. Then we have $\cT\psi = \left(\sum_{n=1}^N \alpha_n\lambda(n) \right) \psi$ and the eigenvalue of the amplifier satisfies
\begin{align}\label{eq: lower bound for amplifier eigenvalue}
     \left|\sum_{n=1}^N \alpha_n\lambda(n)\right|\geq \sum_{\substack{1<p\leq\sqrt{N}\\ (p,q)=1}} \frac{1}{2}\gtrsim_\epsilon N^{1/2-\epsilon}.
\end{align}

\begin{proof}[Proof of \eqref{i3}]
    By Lemma \ref{lem 3.10}, we may bound $\| \sA\psi\|_{L^2(X)}$ instead, and moreover by duality, it suffices to bound $\langle\sA\psi,\phi\rangle_X$ uniformly for any $\phi\in C^\infty(X)$ with $\|\phi\|_{L^{2}(X)}=1$.
     By Proposition \ref{prop: amplification inequality}, we have
     \begin{align}\label{eq: amplification inequality}
         \left|\langle \sA\cT\psi,\phi\rangle_X\right|^2\leq\sum_{m,n\leq N} |\alpha_m  \alpha_n | \sum_{d|(m,n)} \frac{d}{\sqrt{mn}}\sum_{\gamma\in R(mn/d^2)} |I_\sA(\la,\phi,\gamma )|.
     \end{align}
     Fix $0<\epsilon_0\ll\delta_0$. By assuming the support of $k_\la$ and $\Omega$ are sufficiently small, and by Proposition \ref{prop: bound for I_A}, we only need to consider the terms in \eqref{eq: amplification inequality} with $d(g_0^{-1}\gamma g_0,e)\leq1$ and $d(g_0^{-1}\gamma g_0,A)\leq\lambda^{-\delta_0+2\epsilon_0}$. Lemma \ref{lem: Hecke return new} provides a constant $C>0$ so that as long as $n\leq C\lambda^{\delta_0 -2\epsilon_0}$ then \(|\sM(g_0,n,\lambda^{-\delta_0 +2\epsilon_0})|\ll_\epsilon n^\epsilon\).
     Therefore, if $N>0$ is chosen so that  $N^2\leq C\lambda^{\delta_0 -2\epsilon_0}$, by the uniform estimate $|I_\sA(\la,\phi,\gamma )|\lesssim1$ from Proposition \ref{prop: bound for I_A}, we have
     \begin{align*}
         &\sum_{m,n\leq N} |\alpha_m\alpha_n| \sum_{d|(m,n)} \frac{d}{\sqrt{mn}}\sum_{\gamma\in R(mn/d^2)}\left| I_{\sA}(\lambda, \phi, \gamma ) \right|\\
         \lesssim& \sum_{m,n\leq N} |\alpha_m\alpha_n| \sum_{d|(m,n)} \frac{d}{\sqrt{mn}} |\sM(g_0,\frac{mn}{d^2},\lambda^{-\delta_0+2\epsilon_0})|+O_{\epsilon_0,A}(\lambda^{-A})\\
         \lesssim&_\epsilon N^\epsilon\sum_{m,n\leq N} |\alpha_m\alpha_n| \sum_{d|(m,n)} \frac{d}{\sqrt{mn}}+O_{\epsilon_0,A}(\lambda^{-A}).
     \end{align*}
     We have
     \begin{align}\label{eq 323}
     \begin{split}
          \sum_{m,n\leq N} \sum_{d|(m,n)}|\alpha_m\alpha_n|\frac{d}{\sqrt{mn}}&=\sum_{\substack{m,n\leq N\\(m,n)=1}} \sum_{ml,nl\leq N}\sum_{d|l}|\alpha_{ml}\alpha_{nl}|\frac{d}{l\sqrt{mn}}\\
          &\lesssim_\epsilon N^\epsilon\sum_{ml,nl\leq N}\left( \frac{|\alpha_{ml}|^2}{n} +  \frac{|\alpha_{nl}|^2}{m} \right)\lesssim_\epsilon N^\epsilon \sum_{n\leq N} |\alpha_n|^2.
    \end{split}
     \end{align}
     Hence from \eqref{eq: amplification inequality} we have
     \begin{align*}
         \left| \langle \sA\cT\psi,\phi\rangle_X \right|^2   \lesssim&_\epsilon N^\epsilon\sum_{n\leq N} |\alpha_n|^2 +O_{\epsilon_0,A}(\lambda^{-A}).
     \end{align*}
     Combining this with \eqref{eq: moment for alpha} and \eqref{eq: lower bound for amplifier eigenvalue} gives
     \begin{align*}
         N^{1-\epsilon}\left|\langle \sA\psi,\phi\rangle_X \right|^2\lesssim&_\epsilon N^{\frac{1}{2}+\epsilon}+O_{\epsilon,A}(\lambda^{-A}),
     \end{align*}
     which completes the proof by choosing $N=(C\lambda^{\delta_0 -2\epsilon_0})^{\frac{1}{2}}$ and choosing $\epsilon_0$ small.
\end{proof}

\subsection{Proof of (\ref{eq: improved KN sup})}

We recall the amplified pretrace formula \eqref{eq: amplified pretrace}, for any $x,y\in X$,
    \begin{align*}
         \sum_j h_\la(\lambda_j) \cT\psi_j(x) \overline{\cT\psi_j(y)} =\sum_{m,n\leq N} \alpha_m  \overline\alpha_n\sum_{d|(m,n)} \frac{d}{\sqrt{mn}}\sum_{\gamma\in R(mn/d^2)} k_\la(x^{-1}\gamma y).
     \end{align*}
Using the operator $\sA\times\overline\sA$ to act in $x$ and $y$ variables and taking the diagonal, we have 
\begin{equation}\label{l4}
    \begin{aligned}
   \sum_j& h_\la(\lambda_j) |\sA\cT\psi_j(\kappa(u))|^2=\sum_{m,n\leq N} \alpha_m  \overline\alpha_n\sum_{d|(m,n)} \frac{d}{\sqrt{mn}}\sum_{\gamma\in R(mn/d^2)} k_\sA(\lambda,u,\gamma).        
    \end{aligned}
\end{equation}
Here $u\in\tilde\Omega$, $\kappa$ is the map \eqref{eq: diffeo kappa}, and $k_\sA(\lambda,u,\gamma)$ is given by the integral
\begin{equation}\label{eq: integral for kA}
    k_\sA(\lambda,u,\gamma)=\iint_{T^*\tilde\Omega} e^{i\langle u-x,\xi\rangle-i\langle u-y, \eta \rangle }\sA(u,x,\xi)\sA(u,y,\eta) 
        k_\la(\kappa(x)^{-1}\gamma \kappa(y))  dxd\xi dy  d\eta.
\end{equation}

\begin{proposition}\label{prop: amplification inequality supnorm}
    For any $\kappa(u)\in \Omega$, we have
    \[
    \left|\sA\cT\psi(\kappa(u))\right|^2\leq\sum_{m,n\leq N} |\alpha_m  \alpha_n | \sum_{d|(m,n)} \frac{d}{\sqrt{mn}}\sum_{\gamma\in R(mn/d^2)} |k_\sA(\la,u,\gamma)|.
    \]
\end{proposition}
\begin{proof}
    The same as the proof for Proposition \ref{prop: amplification inequality}, we obtain the inequality by dropping all terms on the left-hand side of \eqref{l4} except $\psi$.
\end{proof}

Depending on relations between $u$ and $\gamma$, we bound $k_\sA(\lambda,u,\gamma)$ in the following proposition.

\begin{proposition}\label{p}
Fix $0<\epsilon_0\ll\delta_0$ and let $u\in\tilde\Omega$.
\begin{enumerate}
    \item We have $k_\sA(\lambda,u,\gamma)\lesssim_{\e_0} \la^{\frac{1}{2}} d_g(\kappa(u),\gamma \kappa(u))^{-\frac{1}{2}}$ if
    \begin{align}\label{eq: size condition on dg}
        \la^{-1+2\delta_0+2\epsilon_0}\lesssim d_g(\kappa(u),\gamma \kappa(u))\lesssim1.
    \end{align}
    \item We have $k_\sA(\lambda,u,\gamma)\lesssim_{\e_0} \la^{1-\delta_0+2\e_0}$, if
    \begin{align}\label{eq: size condition on dg1}
        d_g(\kappa(u),\gamma \kappa(u))\lesssim  \la^{-1+2\delta_0+2\epsilon_0}.
    \end{align}
    \item  If $d(\gamma,e)\lesssim 1$ and $d(g_0^{-1}\gamma g_0,A)\gtrsim \lambda^{-\delta_0+2\epsilon_0}$, then $k_\sA(\lambda,u,\gamma)\lesssim_{\epsilon_0,N} \lambda^{-N}$ for any $N>0$.
\end{enumerate}
\end{proposition}
\begin{proof}
    We first recall the estimate \eqref{eq: Schwartz kernel for sA} for the Schwartz kernel of $\sA$, which follows by integration by parts and can be written in the form
    \begin{align}\label{eq: Schwartz kernel for sA 2}
        \int e^{i\langle u-x,\xi\rangle }\sA(u,x,\xi)d\xi  \lesssim_N \la^{2-3\delta_0} (1+\la^{1-\delta_0}|x_1-u_1|)^{-N}(1+\la^{1-2\delta_0}|x_2-u_2|)^{-N})
    \end{align}
    for any $N>0$. Moreover, we can assume $u$ satisfies $u_1\lesssim\la^{-\delta_0}$ and $u_2\lesssim1$, and assume $d(\gamma,e)\lesssim1$, because of the support conditions for $\sA$ and $k_\lambda$. Next we recall the standard pointwise estimate for $k_\lambda$ (see e.g \cite[Lemma 2.8]{Mar16HigerRank}), that is, for any $x,y\in\tilde\Omega$,
    \begin{align}\label{eq: std bound for k}
        k_\la(\kappa(x)^{-1}\gamma\kappa(y))\lesssim\la(1+\la d_g(\kappa(x),\gamma\kappa(y)))^{-\frac{1}{2}}.
    \end{align}
    Using \eqref{eq: Schwartz kernel for sA 2}, it may be seen that
    \begin{equation}\label{eq: sA3}
        \left|\int_{T^*\tilde\Omega}e^{i\langle u-x,\xi\rangle }\sA(u,x,\xi)d\xi dx\right|\lesssim\int \left|\int e^{i\langle u-x,\xi\rangle }\sA(u,x,\xi)d\xi \right|dx \lesssim 1.
    \end{equation}
    Note that by \eqref{eq: Schwartz kernel for sA 2}, we have 
    \begin{equation*}
        \iint e^{i\langle u-x,\xi\rangle-i\langle u-y, \eta \rangle }\sA(u,x,\xi)\sA(u,y,\eta) 
        d\xi  d\eta\lesssim_{\epsilon_0,N} \lambda^{-N}
    \end{equation*}
    if either $d_g(\kappa(u), \kappa(x))\ge \la^{-1+2\delta_0+\epsilon_0} $  or $d_g(\kappa(u), \kappa(y))\ge \la^{-1+2\delta_0+\epsilon_0}$. Therefore, we can restrict $x,y$ in the integral \eqref{eq: integral for kA} such that they satisfy $d_g(\kappa(u), \kappa(x))< \la^{-1+2\delta_0+\epsilon_0} $  and $d_g(\kappa(u), \kappa(y))< \la^{-1+2\delta_0+\epsilon_0}$. Therefore, the condition \eqref{eq: size condition on dg} implies
    \begin{equation*}
        d_g(\kappa(u),\gamma \kappa(u))\sim d_g(\kappa(x),\gamma \kappa(y)),
    \end{equation*}
    so \eqref{eq: std bound for k} gives
    \begin{align*}
        k_\la(\kappa(x)^{-1}\gamma\kappa(y))\lesssim\la(1+\la d_g(\kappa(u),\gamma \kappa(u)))^{-\frac{1}{2}}.
    \end{align*}
    Combining this with \eqref{eq: integral for kA} and \eqref{eq: sA3} proves \textit{(a)}.
    
    To prove \textit{(b)}, first note that by \eqref{eq: Schwartz kernel for sA 2}, the kernel $k_{\sA}(\lambda,u,\gamma)$ is $O(\lambda^{-N})$ if any of the following conditions hold:
\[
|x_1 - u_1| \gtrsim \lambda^{-1+\delta_0+\epsilon_0}, 
\quad \text{or} \quad
|y_1 - u_1| \gtrsim \lambda^{-1+\delta_0+\epsilon_0},
\]
or
\[
d_g(\kappa(u), \kappa(x)) \gtrsim \lambda^{-1+2\delta_0+\epsilon_0},
\quad \text{or} \quad
d_g(\kappa(u), \kappa(y)) \gtrsim \lambda^{-1+2\delta_0+\epsilon_0}.
\]
Therefore, if \eqref{eq: size condition on dg1} holds, we may restrict attention to the region where
\[
|x_1 - y_1| \lesssim \lambda^{-1+\delta_0+\epsilon_0}
\quad \text{and} \quad
d_g(\kappa(x), \gamma \kappa(y)) \lesssim \lambda^{-1+2\delta_0+2\epsilon_0}
\]
If $d_g(\kappa(x), \gamma\kappa(y))\lesssim \la^{-1+\delta_0+\e_0}$, using \eqref{eq: Schwartz kernel for sA 2} and \eqref{eq: std bound for k}, we have, for each fixed $y$,
  \begin{multline*}
        \left|\int_{d_g(\kappa(x),\gamma\kappa(y))\leq \la^{-1+\delta_0+\epsilon_0}}e^{i\langle u-x,\xi\rangle }\sA(u,x,\xi)k_\la(\kappa(x)^{-1}\gamma \kappa(y)) d\xi dx\right|\\
        \lesssim\int_{d_g(\kappa(x),\gamma\kappa(y))\leq \la^{-1+\delta_0+\epsilon_0}} \la^{2-3\delta_0} \la dx \lesssim \la^{1-\delta_0+2\e_0} .
    \end{multline*}
Using \eqref{eq: sA3}, we  have 
  \begin{equation}\label{eq: sA3b}
        \left|\int_{T^*\tilde\Omega}e^{-i\langle u-y,\eta\rangle }\sA(u,y,\xi)d\eta dy\right|\lesssim 1.
    \end{equation}
Thus, we have 
\begin{multline*}
\Bigl|\iint_{d_g(\kappa(x),\gamma\kappa(y))\leq \la^{-1+\delta_0+\epsilon_0}}  e^{i\langle u-x,\xi\rangle-i\langle u-y, \eta \rangle }\sA(u,x,\xi)\\\sA(u,y,\eta) 
        k_\la(\kappa(x)^{-1}\gamma \kappa(y))  dxd\xi dy  d\eta \Bigr|\lesssim \la^{1-\delta_0+2\e_0}.
\end{multline*}
If $d_g(\kappa(x), \gamma\kappa(y))\sim 2^j$ with $\la^{-1+\delta_0+\e_0}\lesssim 2^j\lesssim \la^{-1+2\delta_0+3\e_0}$, then for fixed $y$ and $x_1$, the coordinate \(x_2\) must lie in an interval of the form $|x_2-c(x_1,y)|\lesssim 2^{j} $ for some uniformly bounded function $c$. Thus, by \eqref{eq: Schwartz kernel for sA 2} and \eqref{eq: std bound for k}, for each fixed $y$,
\begin{multline*}
    \left|\int_{d_g(\kappa(x),\gamma\kappa(y))\sim 2^j  }e^{i\langle u-x,\xi\rangle }\sA(u,x,\xi) k(\kappa(x)^{-1}\gamma\kappa(y))d\xi dx\right|\\
    \lesssim\int_{|x_1-y_1|\lesssim \la^{-1+\delta_0+\e_0}}\left(\int_{|x_2-c(x_1,y)|\lesssim 2^{j}  } \la^{2-3\delta_0} \la^{\frac{1}{2}}2^{-\frac{j}{2}} dx_2\right) dx_1\lesssim\la^{\frac{3}{2}-2\delta_0+\e_0}2^{\frac{j}{2}}.
\end{multline*}
By the above and \eqref{eq: sA3b}, we have 
\begin{multline*}
\Bigl|\iint_{d_g(\kappa(x), \gamma\kappa(y))\sim 2^j}  e^{i\langle u-x,\xi\rangle-i\langle u-y, \eta \rangle }\sA(u,x,\xi)\\\sA(u,y,\eta) k_\la(\kappa(x)^{-1}\gamma \kappa(y))  dxd\xi dy  d\eta \Bigr|\lesssim \la^{\frac{3}{2}-2\delta_0+\e_0}2^{\frac{j}{2}}.
\end{multline*}
If we sum over $\la^{-1+\delta_0+\e_0}\lesssim 2^j\lesssim \la^{-1+2\delta_0+2\e_0}$, we obtain \textit{(b)}.

    The proof of \textit{(c)} is the same as Proposition~\ref{prop: bound for I_A} \textit{(b)}, by applying Lemma \ref{lem: Marshall rapid decay}. 
\end{proof}

Because of Proposition \ref{p}, to use the amplification method, we need a new estimate of Hecke returns, which counts how many times the Hecke operators map simultaneously a point and a geodesic segment close to themselves. We define, for $g\in G$, $z\in \cD$, and $0<\delta,\kappa\leq1$,
\begin{align*}
     \sN(g,z,n,\delta,\kappa)=\{\eta\in R(n)\mid d(g^{-1}\eta g,e)\leq1,\; d(g^{-1}\eta g,A)\leq\kappa,\; d_g(z,\eta z)\leq\delta\}.
\end{align*}

\begin{proposition}\label{prop: simultaneous return}
    Let $g\in G$ with $d(g,e)\lesssim1$, $z\in \cD$, and $0\le\delta,\kappa\leq1$. We have 
        \[
        |\sN(g,z,n,\delta,\kappa)|\ll_\epsilon \left(\frac{n}{\kappa\delta}\right)^\epsilon (n\sqrt{\kappa\delta}+1).
        \]
\end{proposition}

 We will prove this proposition in the next section.

\begin{proof}[Proof of \eqref{eq: improved KN sup}]
    Still because of Lemma \ref{lem 3.10}, we may bound $\| \sA\psi\|_{L^\infty(X)}$ instead. By Proposition \ref{prop: amplification inequality supnorm}, we have, for $\kappa(u)\in\supp\sA\psi$,
    \begin{align}\label{eq: supnorm amplified inequality}
        \left|\sA\cT\psi(\kappa(u))\right|^2\leq\sum_{m,n\leq N} |\alpha_m  \alpha_n | \sum_{d|(m,n)} \frac{d}{\sqrt{mn}}\sum_{\gamma\in R(mn/d^2)} |k_\sA(\la,u,\gamma)|.
    \end{align}
    Here $N>0$ is a large parameter, which is some power of $\la$ and will be chosen later, and the amplifier $\cT$ is chosen to be the same as in \ref{subsubsection amplifier}. We will split the sum for $\gamma\in R(mn/d^2)$ in \eqref{eq: supnorm amplified inequality} into different ranges corresponding to those cases in Proposition \ref{p}.  
    
    Fix $0<\epsilon_0\ll\delta_0$. By assuming the support of $k_\la$ and $\Omega$ are sufficiently small, we assume that $k_\sA(\la,u,\gamma)\neq0$ only if $d_g(\kappa(u),\gamma\kappa(u))\leq1$ and $d(g_0^{-1}\gamma g_0,e)\leq1$.   
    Note that, by Proposition \ref{p} \textit{(c)}, if $d(g_0^{-1}\gamma g_0,A)\geq \la^{-\delta_0+2\epsilon_0}$, then $k_\sA(\la,u,\gamma)$ decays rapidly, so for $m,n\leq N$, $d|(m,n)$, we have
    \begin{align}\label{eq: supnorm range 2}
        \sum_{\substack{\gamma\in R(mn/d^2)\\d(g_0^{-1}\gamma g_0,A)\geq \la^{-\delta_0+2\epsilon_0}}} |k_\sA(\la,u,\gamma)|\lesssim_{\e_0,A}\la^{-A}.
    \end{align}
    Hence we can restrict $\gamma$ to the smaller region $d(g_0^{-1}\gamma g_0,A)< \la^{-\delta_0+2\epsilon_0}$. If $d_g(\kappa(u),\gamma\kappa(u))\leq \la^{-1+2\delta_0+2\epsilon_0}$, then we have
    \begin{align}\label{eq: supnorm range 1a}
    \begin{split}
        &\sum_{\substack{\gamma\in R(mn/d^2)\\d_g(\kappa(u),\gamma\kappa(u))\leq \la^{-1+2\delta_0+2\epsilon_0}\\d(g_0^{-1}\gamma g_0,A)<\la^{-\delta_0+2\epsilon_0}}} |k_\sA(\la,u,\gamma)|\\
        \lesssim&_{\e_0} \la^{1-\delta_0+2\e_0} \left|\sN(g_0,\kappa(u),\frac{mn}{d^2},\la^{-1+2\delta_0+2\epsilon_0},\la^{-\delta_0+2\epsilon_0})\right|\\
        \lesssim&_\e  N^\e\la^{\e+4\e_0}  \left( \frac{mn}{d^2}\la^{\frac{1-\delta_0}{2}}+\la^{1-\delta_0}\right),
    \end{split}
    \end{align}
    by applying \textit{(b)} in Proposition \ref{p} and Proposition \ref{prop: simultaneous return}.

    The last range is
    \begin{align}\label{eq: supnorm range 3}
        &\sum_{\substack{\gamma\in R(mn/d^2)\\
        \la^{-1+2\delta_0+2\epsilon_0}< d_g(\kappa(u),\gamma\kappa(u))\leq1
        \\d(g_0^{-1}\gamma g_0,A)< \la^{-\delta_0+2\epsilon_0}}} |k_\sA(\la,u,\gamma)|=\sum_{k=1}^{(1-2\delta_0-2\epsilon_0)\log_2\la} I_k.
    \end{align}
    Here we divide the sum further dyadically, and $I_k$ is the sum of $|k_\sA(\la,u,\gamma)| $ over $\gamma\in R(mn/d^2)$ satisfying $d_g(\kappa(u),\gamma\kappa(u))\in(\la^{-1+2\delta_0+2\epsilon_0}2^{k-1}, \la^{-1+2\delta_0+2\epsilon_0}2^{k}]$ and $d(g_0^{-1}\gamma g_0,A)< \la^{-\delta_0+2\epsilon_0}$. By Proposition \ref{p} and Proposition \ref{prop: simultaneous return} \textit{(a)}, we obtain
    \begin{align*}
        I_k &\lesssim |\sN(g_0,\kappa(u),\frac{mn}{d^2},\la^{-1+2\delta_0+2\epsilon_0}2^{k},\la^{-\delta_0+2\epsilon_0})|\left(\la^{\frac{1}{2}}\left(\la^{-1+2\delta_0+2\epsilon_0}2^{k}\right)^{-\frac{1}{2}}\right)\\
        &\lesssim_\e N^\e\la^\e\left(\frac{mn}{d^2} \la^{\frac{-1+\delta_0}{2}+2\e_0}2^{\frac{k}{2}} +1 \right)\la^{1-\delta_0-\e_0}2^{-\frac{k}{2}}\\
        &= N^\e\la^\e\left(\frac{mn}{d^2} \la^{\frac{1-\delta_0}{2}+\e_0} + \la^{1-\delta_0-\e_0}2^{-\frac{k}{2}} \right),
    \end{align*}
    which shows that \eqref{eq: supnorm range 3} is bounded by
    \begin{align*}
        \sum_{k=1}^{(1-2\delta_0-2\epsilon_0)\log_2\la} N^\e\la^\e\left(\frac{mn}{d^2} \la^{\frac{1-\delta_0}{2}+\e_0} + \la^{1-\delta_0-\e_0}2^{-\frac{k}{2}} \right)\lesssim_\e N^\e \la^\e \left(\frac{mn}{d^2} \la^{\frac{1-\delta_0}{2}+\e_0} + \la^{1-\delta_0-\e_0}\right).
    \end{align*}
    Thus, by combining this with  \eqref{eq: supnorm range 2} and \eqref{eq: supnorm range 1a}, and choosing $\e_0$ arbitrarily small, we have, for $m,n\leq N$, $d|(m,n)$,
    \begin{align*}
        \sum_{\gamma\in R(mn/d^2)} |k_\sA(\la,u,\gamma)|\lesssim N^\e \la^\e \left(\frac{mn}{d^2} \la^{\frac{1-\delta_0}{2}} + \la^{1-\delta_0}\right).
    \end{align*}
    Using \eqref{eq: supnorm amplified inequality}, we have
    \begin{align}\label{eq 335}
        \left|\sA\cT\psi(\kappa(u))\right|^2\lesssim_\e N^\e \la^\e\sum_{m,n\leq N} |\alpha_m  \alpha_n | \sum_{d|(m,n)} \left(\frac{\sqrt{mn}}{d} \la^{\frac{1-\delta_0}{2}} + \frac{d}{\sqrt{mn}} \la^{1-\delta_0}\right).
    \end{align}
    Recall that by \eqref{eq 323}, we have
    \begin{align*}
          \sum_{m,n\leq N} \sum_{d|(m,n)}|\alpha_m\alpha_n|\frac{d}{\sqrt{mn}}\lesssim_\epsilon N^\epsilon \sum_{n\leq N} |\alpha_n|^2.
     \end{align*}
    We also have
     \begin{align*}
         \sum_{m,n\leq N}  \sum_{d|(m,n)}|\alpha_m\alpha_n|\frac{\sqrt{mn}}{d}\lesssim_\epsilon N^{1+\epsilon} \left( \sum_{n\leq N} |\alpha_n|\right)^2.
     \end{align*}
     Hence from \eqref{eq 335} we have
     \begin{align*}
        \left|\sA\cT\psi(\kappa(u))\right|^2\lesssim_\e N^\e \la^\e\left( N\la^{\frac{1-\delta_0}{2}}\left( \sum_{n\leq N} |\alpha_n|\right)^2  + \la^{1-\delta_0}\sum_{n\leq N} |\alpha_n|^2\right).
    \end{align*}
    Combining this with \eqref{eq: moment for alpha} and \eqref{eq: lower bound for amplifier eigenvalue} gives
    \begin{align*}
        N^{1-\e}\left|\sA\psi(\kappa(u))\right|^2\lesssim_\e N^\e \la^\e\left( N^2\la^{\frac{1-\delta_0}{2}}  + N^{\frac{1}{2}}\la^{1-\delta_0}\right).
    \end{align*}
    Choosing $N=\la^{\frac{1-\delta_0}{3}}$ gives the desired bound $|\sA\psi(\kappa(u))|\lesssim_\e \la^{\frac{5}{12}(1-\delta_0)+\e}.$
\end{proof}

\section{Counting}\label{sec: counting}
To prove Proposition \ref{prop: simultaneous return}, we first need a Diophantine lemma, which slightly modifies the corresponding Diophantine lemmas \cite[Lemma 1.4]{IS95} and \cite[Lemma 3.2]{Mar16}.
\begin{lemma}\label{Diophantine lemma}
    Let $A,B>0$. Let $1\gtrsim y\ge A$, $0< \kappa\leq 1$, and $0< \delta\leq 1$. Then we have
    \begin{align}\label{eq: ISDL}
        \left| \left\{ (r,s)\in\BZ^2\mid  |r^2+ys^2 -n|\leq Bn\delta,\; s^2\leq Bn\kappa\right\} \right|\lesssim \left( \frac{n}{\delta}\right)^\e\left( n \left((\delta\kappa)^{\frac{1}{2}}+\delta\right)  +1\right),
    \end{align}
    and
    \begin{align}\label{eq: MDL}
        \begin{split}\left| \left\{ (r,s)\in\BZ^2\mid  |r^2-ys^2 -n|\leq Bn\kappa,\;r^2\leq Bn ,\;s^2\leq Bn\delta\right\} \right|\\
        \lesssim \left( \frac{n}{\kappa}\right)^\e\left( n \left((\delta\kappa)^{\frac{1}{2}}+\kappa\right)  +1\right),\end{split}
    \end{align}
    for all $\e>0$, where the implied constants depend only on $A$, $B$, and $\e$.
\end{lemma}

\begin{proof}
    We first prove \eqref{eq: ISDL}. For $Q\geq1$, we can find coprime integers $p,q$ with $1\leq q\leq Q$ such that
    \[
    \left|\frac{p}{q}-y\right|\leq\frac{1}{qQ}.
    \]
    Thus $|r^2+ys^2 -n|\leq Bn\delta$ implies
    \begin{align*}
        |qr^2+ps^2-qn|\leq q\left( |r^2+ys^2 -n|+\left|\frac{p}{q}-y\right||s|^2\right)\leq Bn\delta q+\frac{Bn\kappa}{Q}.
    \end{align*}
    Here we have also used $|s|^2\leq Bn\kappa$. We choose $Q=\left(\frac{\kappa}{\delta}\right)^{\frac{1}{2}} + A^{-1}$. This choice implies $p\geq1$ since $y\geq A$, and gives
    \begin{align*}
        |qr^2+ps^2-qn|\leq 2Bn \left(\kappa\delta\right)^{\frac{1}{2}}+ \frac{Bn\delta}{A}\leq C n \left((\delta\kappa)^{\frac{1}{2}}+\delta\right),\quad\text{ where }C=2B+\frac{B}{A}.
    \end{align*}
    Hence, the number of pairs $(r, s)$ in \eqref{eq: ISDL} is bounded by
    \begin{align}\label{eq: IS sum}
        \sum_{|m-qn|\leq C n  \left((\delta\kappa)^{\frac{1}{2}}+\delta\right)} | \left\{ (r,s)\mid qr^2+ps^2=m  \right\} |.
    \end{align}
    For each $m$, $| \left\{ (r,s)\mid qr^2+ps^2=m  \right\} |\lesssim_\e m^\e$. To see this, we can moreover assume that $q$ and $p$ are square-free by incorporating the square parts of $q$ and $p$ into $r$ and $s$ respectively. With this done a solution of $qr^2+ps^2=m$ gives an algebraic integer $rq+s\sqrt{-pq}\in \BQ(\sqrt{-pq})$ of norm $qm$. The latter is bounded by $O_\e((qm)^\e)$ using the finiteness of the units in $\BQ(\sqrt{-pq})$ and the bound for the divisor function. It follows that the \eqref{eq: IS sum} has a bound of the form as the right-hand side of \eqref{eq: ISDL}.

    To prove \eqref{eq: MDL}, we choose $Q=\left(\frac{\delta}{\kappa}\right)^{\frac{1}{2}} + A^{-1}$ and similarly choose the corresponding $p,q,C$, so that the number of pairs $(r, s)$ in \eqref{eq: MDL} can be bounded by
    \begin{align}\label{eq: M sum}
        \sum_{|m-qn|\leq C n  \left((\delta\kappa)^{\frac{1}{2}}+\kappa\right)} | \left\{ (r,s)\mid qr^2-ps^2=m,\;r^2\leq Bn ,\;s^2\leq Bn  \right\} |.
    \end{align}
    Let $\cO$ denote the ring of integers in the field $\BQ(\sqrt{pq})$, let $N$ denote the norm map, and let $|\cdot|_1$ and $|\cdot|_2$ denote its two archimedean valuations. Every solution of $qr^2-ps^2=m$ determines an element $rq+s\sqrt{pq}\in\cO$, whose norm is $qm$. We therefore have
    \begin{multline*}
        | \left\{ (r,s)\mid qr^2-ps^2=m,\;r^2\leq Bn ,\;s^2\leq Bn  \right\} |\\
        \leq |\{z\in\cO\mid Nz=qm,\; |z|_1+|z|_2\leq 2\sqrt{Bn}(q+\sqrt{pq}) \}|.
    \end{multline*}
    The set on the right-hand side maps to the set of ideals in $\cO$ with norm $qm$, and the number of such ideals is bounded by $O_\e((qm)^\e)$ by bounding the divisor function. The fibers of this map are orbits under multiplication by units, and the condition
    \[|z|_1+|z|_2\leq 2\sqrt{Bn}(q+\sqrt{pq})\lesssim n\delta^{-\frac{1}{2}}\]
    implies that the fibers have size $\lesssim_\e (n/\kappa)^\e$. Summing this bound in \eqref{eq: M sum} completes the proof.
\end{proof}

\begin{proof}[Proof of Proposition \ref{prop: simultaneous return}]
Recall that we have $g\in G$ with $d(g,e)\lesssim1$, and $z\in \cD$. Let $[\alpha,\beta,\gamma]$ be the quadratic form associated with the infinite-length geodesic in $\BH$ corresponding to $g_0 A g_0^{-1}$ (which is a semicircle or a vertical line), so that
\begin{align}\label{eq: discrimiant 1}
    \beta^2-4\alpha\gamma=1
\end{align}
and the roots of $\alpha z^2+\beta z+ \gamma$ are the endpoints of that geodesic. Note that the case when one endpoint is at $i\infty$ corresponds to $\alpha=0$. We may parametrize $g A g^{-1}$ as
\begin{align}\label{eq: parametrization of A}
    g A g^{-1}=\left\{\begin{pmatrix}
        t-\beta u& -2\gamma u\\
        2\alpha u& t+\beta u
    \end{pmatrix} \mid t^2-u^2=1, t>0  \right\}.
\end{align}
By the boundedness of $g$, we can assume that $\alpha,\beta,\gamma\lesssim 1$. We use another quadratic form $[\tilde\alpha,\tilde\beta,\tilde\gamma]$ with
\begin{align}\label{eq: discrimiant -1}
    \tilde\beta^2-4\tilde\alpha\tilde\gamma=-1
\end{align}
be the quadratic form associated with $z$, that is, $z$ is the unique solution of $\tilde\alpha z^2+ \tilde\beta z+ \tilde\gamma$ in $\BH$. Similarly, we have $\tilde\alpha,\tilde\beta,\tilde\gamma\lesssim1$, and the stabilizer $K_z$ of $z$ in $\SL(2,\BR)$ can be parametrized as
\begin{align}\label{eq: parametrization of K}
    K_z=\left\{ \begin{pmatrix}
        t-\tilde\beta u & 2\tilde\gamma u\\
        2\tilde\alpha u& t+\tilde\beta u 
    \end{pmatrix} \mid t^2+u^2=1 \right\}.
\end{align}
If $\eta\in R(n)$, under the embedding $\iota$, defined by \eqref{eq: embedding of quaternion}, we can write
\begin{align*}
    \iota(\eta) = \begin{pmatrix}
        x_0+x_1\sqrt{a}&b(x_2+x_3\sqrt{a})\\
        x_2-x_3\sqrt{a}& x_0-x_1\sqrt{a}
    \end{pmatrix}
\end{align*}
with $\det\iota(\eta)=n$, $Ex_i\in \BZ$ for some fixed positive integer $E$ (which is determined from the order $R$). If $\eta\in R(n)$ satisfies $d_g(z,\eta z)\leq\delta$ and $d(g^{-1}\eta g,A)\leq\kappa$, by \eqref{eq: parametrization of A} and \eqref{eq: parametrization of K}, there exist $t_1,u_1, t_2,u_2$ such that $0<t_1\lesssim1$, $u_1\lesssim1$, $t_1^2-u_1^2=1$, $t_2^2+u_2^2=1$, and
\begin{align}
    \frac{x_0}{\sqrt{n}} &= t_1+ O(\kappa)=t_2+O(\delta),\label{eq: x0}\\
    \frac{x_1\sqrt{a}}{\sqrt{n}}&=-\beta u_1+O(\kappa)=-\tilde\beta u_2+O(\delta),\label{eq: x1}\\
    \frac{x_2}{\sqrt{n}}&=\left(\alpha-\frac{\gamma}{b}\right) u_1 + O(\kappa)=\left(\tilde\alpha+\frac{\tilde\gamma}{b}\right) u_2 + O(\delta),\notag\\
    \frac{x_3\sqrt{a}}{\sqrt{n}}&=-\left(\alpha+\frac{\gamma}{b}\right) u_1 + O(\kappa)=-\left(\tilde\alpha-\frac{\tilde\gamma}{b}\right) u_2 + O(\delta).\notag\
\end{align}

Let us first assume $\kappa>\delta$.
Note that, by \eqref{eq: discrimiant -1}, at least one of $|\tilde\beta|$ and $|\tilde\alpha\pm\frac{\tilde\gamma}{b}|$ must be $\ge \frac{1}{10|b|}$. 
Let us assume $|\tilde\beta|\ge \frac{1}{10|b|}$, as the other two cases can be treated similarly. 
From \eqref{eq: x0} and \eqref{eq: x1}, we have 
\begin{equation}\label{l1}
    1=t_2^2+u_2^2=\frac{x_0^2}{n}+\frac{x_1^2a}{n\tilde \beta^2}+O(\delta).
\end{equation}
If $|\beta|\le \kappa$, then by \eqref{eq: x1}, 
$|x_1|\lesssim \sqrt{n}\kappa\leq\sqrt{n\kappa}$. If $|\beta|> \kappa$, note that by \eqref{eq: x0} and \eqref{eq: x1}, we also have 
\begin{equation}\label{l2}
    1=t_1^2-u_1^2=\frac{x_0^2}{n}-\frac{x_1^2a}{n \beta^2}+O(\kappa|\beta|^{-1}).
\end{equation}
Using \eqref{l1} and \eqref{l2}, we have 
\begin{equation*}
    0=\frac{x_1^2a}{n\tilde \beta^2}+\frac{x_1^2a}{n \beta^2} +O(\kappa|\beta|^{-1})+O(\delta)\Rightarrow |x_1|\lesssim \sqrt{n\kappa}.
\end{equation*}
So for any values of $\beta$, we have 
 $|x_1|\lesssim \sqrt{n\kappa}$. 
If we combine this with \eqref{l1} and applying \eqref{eq: ISDL}, the number of choices
of $x_0,x_1$ is at most
\begin{align}\label{eq: bd for x0 x1}
    \left( \frac{n}{\delta}\right)^\e\left( n \left((\delta\kappa)^{\frac{1}{2}}+\delta\right)  +1\right)\lesssim \left( \frac{n}{\delta}\right)^\e\left( n (\delta\kappa)^{\frac{1}{2}} +1\right).
\end{align}
Using $\det\iota(\eta)=n$, for each of these $x_2, x_3$ must satisfy
\begin{align}\label{eq: det=n}
    -bx_2^2+ab x_3^2=n-x_0^2+ax_1^2,
\end{align}
and we have $x_j\lesssim\sqrt{n}$. By working with the field $\BQ(\sqrt{a})$ in the proof of Proposition \ref{Diophantine lemma}, we see that the number of choices for $x_2$ and $x_3$ is $O_\e(n^\e)$.

If $\kappa\leq\delta$, similarly, at least one of $|\beta|$ and $|\alpha\pm\frac{\gamma}{b}|$ must be $\ge \frac{1}{10|b|}$ by \eqref{eq: discrimiant 1}, so we may just assume $|\beta|\ge \frac{1}{10|b|}$ in our proof. From \eqref{eq: x0} and \eqref{eq: x1}, we have 
\begin{equation}\label{l1'}
    1=t_1^2-u_1^2=\frac{x_0^2}{n}-\frac{x_1^2a}{n \beta^2}+O(\kappa).
\end{equation}
By discussing the size of $|\tilde\beta|$ similarly as in the previous case, we can obtain $|x_1|\lesssim\sqrt{n\delta}$. By combining this with \eqref{l1'} and applying \eqref{eq: MDL}, the number of choices of $x_0,x_1$ is at most
\begin{align*}
    \left( \frac{n}{\kappa}\right)^\e\left( n \left((\delta\kappa)^{\frac{1}{2}}+\kappa\right)  +1\right)\lesssim \left( \frac{n}{\kappa}\right)^\e\left( n (\delta\kappa)^{\frac{1}{2}} +1\right).
\end{align*}
The number of choices of $x_2,x_3$ can be bounded by $O(n^\e)$ using \eqref{eq: det=n} by the same argument as in the previous case.
\end{proof}

\bibliographystyle{alpha}
\bibliography{refs}

\end{document}